\documentclass[10pt]{article}
\DeclareMathSymbol{\varChi}{\mathord}{letters}{88}

\usepackage[latin1]{inputenc}
\usepackage{amsfonts}
\usepackage{mathrsfs}
\usepackage{amsmath}
\usepackage{amssymb}

\usepackage{amsthm}
\usepackage{tensor}
 \usepackage{color}
\definecolor{red}{rgb}{1.0,0.0,0.0}

\definecolor{blu}{rgb}{0.0,0.0,1.0}

\definecolor{gre}{rgb}{0.03,0.50,0.03}

\definecolor{white}{rgb}{1,1,1}
\def\white#1{{\textcolor{white}{#1}}}

\usepackage[T1]{fontenc}
\newcommand{\ud}{\,\mathrm{d}}
\def\K{\mathbb K}
\def\M{\mathbb M}
\def\R{\mathbb R}
\def\X{\mathbb X}
\def\Y{\mathbb Y}
\def\B{\mathbb B}
\def\N{\mathbb N}
\def\W{\mathbb W}
\def\Z{\mathbb Z}
\def\A{\mathbb A}
\def\S{\mathbb S}
\def\V{\mathbb V}

\def \shd{\mathcal D}
\def \shl{\mathcal L}
\def \shp{\mathcal P}
\newtheorem{Theorem}{Theorem}[section]
\newtheorem{Definition}[Theorem]{Definition}
\newtheorem{Proposition}[Theorem]{Proposition}
\newtheorem{Lemma}[Theorem]{Lemma}
\newtheorem{Corollary}[Theorem]{Corollary}

\newtheorem{Remark}[Theorem]{Remark}
\newtheorem{Example}[Theorem]{Example}
\newtheorem{Notation}[Theorem]{Notation}

\selectfont

\author{Giorgio Fabbri\footnote{Aix-Marseille Univ. (Aix-Marseille School of Economics), CNRS \&
EHESS. Centre de la vieille Charit\'e, 2 rue de la Charit\'e, 13002 Marseille, France. E-mail: giorgio.fabbri@univ-amu.fr}
 \; and \;  Francesco Russo\footnote{ENSTA ParisTech, Universit\'e Paris-Saclay, Unit\'e de Math\'ematiques appliqu\'ees, 828, boulevard des Mar\'echaux,
F-91120 Palaiseau.
E-mail: francesco.russo@ensta-paristech.fr} }
\title{Infinite dimensional weak Dirichlet processes and
convolution type processes.}
\date{June 10th 2016}

\begin{document}

\maketitle

\begin{abstract}
The present paper continues the study of infinite dimensional calculus via regularization, started by C. Di Girolami and the second named author,  introducing the notion of \emph{weak Dirichlet process}  in this context. Such a process $\X$, taking values in a Banach space $H$, is the sum of a local martingale and a suitable {\it orthogonal} process. \\ 
The concept of weak Dirichlet process fits the notion of \emph{convolution type processes}, a class
 including mild solutions for  stochastic evolution equations on infinite dimensional Hilbert spaces and in particular of 
several classes of  stochastic partial differential equations (SPDEs).\\
In particular the mentioned decomposition  appears to be a  substitute of an It\^o's type formula  applied to
 $f(t, \X(t))$ where $f:[0,T] \times H \rightarrow \R$  is a $C^{0,1}$ function and $\X$ a convolution type processes.

\end{abstract}

\bigskip


\medskip

{\bf 2010 AMS Math Classification:}    60H05, 60H15, 60H30, 35R60.

\bigskip

\section{Introduction}
Stochastic calculus via regularization for real processes was initiated in \cite{RussoVallois91} and \cite{RussoVallois93}.
 It is an efficient calculus for non-semimartingales whose related literature is  surveyed in \cite{RussoVallois07}. 
We briefly recall some essential notions from \cite{RussoVallois91} and \cite{RussoVallois93}. Let $T > 0$ and $\left (\Omega , \mathscr{F}, \mathbb{P} \right )$ be a complete probability space. Given $Y$ (respectively $X$) an a.s. bounded (respectively continuous) real process\footnote{In the whole paper, all the considered processes are supposed to be measurable from $[0,T] \times \Omega$ (equipped with the product of the Borel $\sigma$-field of $[0,T]$ and the $\sigma$-algebra of all events $\mathscr{F}$) and the value space, for instance $\R$, endowed with the Borel $\sigma$-field.} 
defined on $[0,T]$, the forward integral of $Y$ with respect to $X$  and the covariation between $Y$ and $X$ are defined as follows.
Suppose that, for every $t \in [0,T]$, the limit $I(t)$ (respectively $C(t)$) in probability exists:
\begin{eqnarray} \label{DefRealIntCov}
I(t) : &=& \lim_{\epsilon \to 0^+} \int_0^t Y(r) \left ( \frac{X(r+\epsilon) - 
X(r)}{\epsilon} \right ) \ud r,
\nonumber\\
& &  \\
C(t) :&=& \lim_{\epsilon \to 0^+} \int_{0}^{t} \frac{ \left (X({r+\epsilon})-
X({r}) \right )
\left (Y({r+\epsilon})-Y({r})\right )}{\epsilon}dr.   
\nonumber
\end{eqnarray}
If the random process $I$ (respectively $C$) admits a continuous version,  this is denoted  by  $\int_0^\cdot Y d^-X$ (respectively $[X,Y]$).  It is the {\it  forward integral} of $Y$ with respect to $X$ (respectively the {\it covariation} of $X$ and $Y$). If $X$ is a real continuous semimartingale and $Y$ is a c\`adl\`ag process which is progressively measurable (respectively  a semimartingale), the integral $\int_0^\cdot Y d^-X$  (respectively the covariation $[X,Y]$) is the same as the classical It\^o's integral (respectively covariation).

Real  processes  $X$ for which $[X,X]$ exists are called  {\it finite quadratic variation  processes}.
A rich class of finite quadratic variation processes is provided by Dirichlet processes.  Let $(\mathscr{F}_t, t \in [0,T])$ be a fixed filtration, fulfilling the usual conditions. A real process $X$ is said to be  {\it Dirichlet} (or {\it F\"ollmer-Dirichlet}) if it is the sum of a local martingale $M$ and a  {\it zero quadratic variation process} $A$, i.e. such that $[A,A] = 0$. Those processes were defined by H. F\"ollmer \cite{FoDir} 
using  limits of discrete sums.  A significant generalization, due to \cite{errami, GozziRusso06}, is the notion
of {\it weak Dirichlet process}, extended to the case of jump processes in \cite{coquet}.

\begin{Definition}
\label{def:Dirweak}
A real process $X \colon [0,T]\times \Omega \to \mathbb{R}$ is called weak Dirichlet process if it can be written as $X=M+A$,
where $M$ is a local martingale and $A$ is a process such that $\left[ A,N\right] =0$ for every continuous local martingale $N$ and $A(0)=0$.
\end{Definition}
Obviously a semimartingale  is a weak Dirichlet process. It can be proved (see Remark 3.5 of \cite{GozziRusso06}) that the decomposition described in Definition \ref{def:Dirweak} is unique.

Elements of calculus via regularization were extended to (real) Banach space valued
processes in a series of papers, see e.g. \cite{DiGirolamiRusso09, DGRNote, DiGirolamiRusso11Fukushima, DiGirolamiRusso11}. 
Two classical notions of stochastic calculus in Banach spaces, which appear in \cite{MetivierPellaumail80} and \cite{Dinculeanu00} are 
the scalar\footnote{The two mentioned monographs use the  term {\it real} instead of {\it scalar}; we changed it  to avoid confusion with the quadratic variation of real processes.} and tensor quadratic variations. We propose here, following \cite{DiGirolamiRusso11}, a definition of them using the regularization approach, even though, originally they appeared in a discretization framework.

\begin{Definition}
\label{def:per-nota}
  Consider a real separable Banach spaces $B$.
We say that a  process $\X\colon [0,T]\times \Omega \to B$,
a.s. square integrable,
 admits a {\bf scalar quadratic variation} if, for any $t\in [0,T]$,
 the limit, for $\epsilon\searrow 0$ of 
\[
[\X,\X]^{\epsilon, \mathbb{R}}(t):= \int_{0}^{t} \frac{\left 
|\X({r+\epsilon})-\X({r}) \right |^2_B}{\epsilon} dr
\]
exists in probability and it admits a continuous version. The limit process 
is called scalar quadratic variation of $\X$ and it is denoted  by  $[\X,\X]^{\mathbb{R}}$. 
\end{Definition}

\begin{Definition}
\label{def:tensorcovariation}
Consider two real separable Banach spaces $B_1$ and $B_2$. Suppose that either $B_1$ or $B_2$ is different from $\mathbb{R}$. Let  $\X\colon [0,T]\times \Omega \to B_1$ and  $\Y\colon [0,T]\times \Omega \to B_2$  be two
a.s. square integrable processes.
We say that $(\X,\Y)$ admits a {\bf tensor covariation} if the limit, for $\epsilon\searrow 0$ of the $B_1\hat\otimes_\pi B_2$-valued processes
\[
[\X,\Y]^{\otimes, \epsilon}:= \int_{0}^{\cdot} \frac{ \left (\X({r+\epsilon})-\X({r}) \right ) \otimes
\left (\Y({r+\epsilon})-\Y({r})\right )}{\epsilon}dr,
\]
exists ucp\footnote{Given a Banach space $B$ and a probability space 
$(\Omega, \mathbb{P})$ a family of processes $\X^\epsilon\colon \Omega
\times [0, T] \to B$ is said to converge in the ucp sense to
$\X\colon \Omega \times [0, T] \to B$, when $\epsilon$ goes to zero,
if $
\lim_{\epsilon \to 0} 
\sup_{t\in [0,T]} |\X^\epsilon_t - \X_t|_B = 0,
$
in probability. Observe that we use the convergence in probability in the definition of covariation when the two processes are real and ucp convergence, when either $\X$ or $\Y$ is not one-dimensional. When $\X=\Y$ the two definitions are equivalent (see Lemma 2.1 of \cite{RussoVallois07}).}.
 The limit process is called {\bf tensor covariation} of $(\X,\Y)$ 
and is denoted  by  $[\X,\Y]^\otimes$. The tensor covariation $[\X,\X]^\otimes$ is called {\bf tensor quadratic variation} of $\X$ and
is denoted  by  $[\X]^\otimes$.
\end{Definition}

The concepts of scalar and  tensor quadratic variation are too strong in certain contexts: several interesting examples of Banach (or even Hilbert) space valued processes have no tensor quadratic variation. For this reason Di Girolami and Russo introduced (see for instance Definition 3.4 of \cite{DiGirolamiRusso11Fukushima}) the notion of $\chi$-covariation. It is recalled in Definition \ref{def:covariation} and widely use in this work. 
Their idea was to introduce a suitable space $\chi$ continuously embedded into the dual of the  projective tensor space $B_1 \hat \otimes_\pi B_2,$ called \emph{Chi-subspace} and to introduce the notion of \emph{$\chi$-quadratic variation} (and of $\chi$-covariation, denoted by $[\cdot,\cdot]_{\chi}$), recalled in  Section \ref{sub4.2}.

An interesting example of Banach space valued process with no tensor quadratic variation is the $C([-\tau,0])$-valued process $\X$ defined as the  frame (or window) of a standard Brownian motion (see \cite{DiGirolamiRusso11}).


A second example, that indeed constitutes a main motivation for the present paper, is given by mild solutions of stochastic evolution equations in infinite dimensions: they have no scalar quadratic variation even if driven by a one-dimensional Brownian motion. We briefly recall their definition.



Consider a stochastic evolution equations of the form
\begin{equation}\label{eq:SPDEIntro}
\left \{
\begin{array}{l}
\ud \X(t) = \left ( A\X(t)+ b_{0}(t,\X(t)) \right ) \ud t + \sigma_{0}(t,\X(t)) \ud
 \W_Q(t)\\[5pt]
\X(0)=x,
\end{array}
\right.
\end{equation}
characterized by a generator of a $C_0$-semigroup $A$, Lipschitz coefficients $b$ and $\sigma$ and a $Q$-Wiener process (with respect to some covariance operator $Q$) $\W$.

As described for example in Part III of \cite{DaPratoZabczyk96}, several families of partial differential equations with stochastic forcing terms or coefficients (SPDEs) can be reformulated as stochastic evolution equations and then can be studied in the general abstract setting; of course for any of them the specification of the generator $A$ and of the functions $b_{0}$ and $\sigma_{0}$ are different. Examples of SPDEs that can be rewritten in the form \eqref{eq:SPDEIntro} (see e.g. \cite{DaPratoZabczyk96} Part III) are stochastic heat (and more general parabolic) equations with Dirichlet or Neumann boundary conditions, wave equations, delay equations, reaction-diffusion equations.

There are some different possible ways to define what we mean by \emph{solution} of (\ref{eq:SPDEIntro}), among them there is the notion of \emph{mild solution}. 
A progressively measurable process $\X(t)$ is a mild solution of (\ref{eq:SPDEIntro}) (see \cite{DaPratoZabczyk92} Chapter 7 or \cite{GawareckiMandrekar10} Chapter 3) if 
it is the solution of the following integral equation
\[
\X(t) = e^{tA}x + \int_0^t e^{(t-r)A} b_{0}(r,\X(r)) \ud r + \int_0^t e^{(t-r)A} \sigma_{0}(r,\X(r)) \ud \W_Q(r).
\]
This concept is widely used in the literature. Mild solutions are particular cases of convolution type processes defined in Section \ref{sec:SPDEs}.


\subsection*{The contributions of the work}

The novelty of the present paper arises both at the level of the stochastic calculus and of the infinite dimensional stochastic differential equations (and more in general convolution type processes).

The stochastic calculus part starts (Sections \ref{SecStochInt}) with a natural extension (Definition \ref{def:forward-integral}) of the notion of forward integral in Banach spaces introduced in \cite{DiGirolamiRusso09} and with the proof of its equivalence with the classical notion of integral when we integrate a predictable process with respect to a local martingale (Theorem \ref{th:integral-forrward=ito-martingale}). We also prove (Proposition \ref{P25}) that, under suitable hypotheses,  forward integrals also extend Young type integrals.

In Section
 \ref{SecTensor}, we extend  the notion of Dirichlet process to 
infinite dimensions. According to the literature, an Hilbert space-valued stochastic process can be 
naturally  considered to be an {\it infinite dimensional Dirichlet process} if it is 
the sum of a local martingale  and a zero energy process.
A {\it zero energy process} (with some light sophistications) is a process such that 
the expectation of the quantity in Definition \ref{def:per-nota}
converges to zero when $\varepsilon$ goes to zero. This happens for instance
in \cite{Denis},  even  though that decomposition also appears
in \cite{maroeck} Chapter VI Theorem 2.5, for processes associated with 
an infinite-dimensional Dirichlet form.

Extending  F\"ollmer's notion of Dirichlet
process to infinite dimensions, a process $\X$ taking values in 
a Hilbert space $H$, could be called
{\it Dirichlet} if it is the sum of a local martingale $\M$
plus a process $\A$ having a zero scalar quadratic variation. However that natural notion is not suitable 
for an  efficient stochastic calculus for  infinite dimensional
 stochastic differential equations. 
Indeed solutions of SPDEs are in general no of such form, so this prevents to use It\^o type formulae
or generalized Doob-Meyer decompositions.

Using the notion of $\chi$-finite quadratic variation process introduced in \cite{DiGirolamiRusso11Fukushima} we introduce the notion of $\chi$-Dirichlet process as the sum of a 
local martingale $\M$ and a process $\A$ having a zero
 $\chi$-quadratic variation.

A completely new notion in the present paper 
is the one of  Hilbert valued {\it $\nu$-weak Dirichlet process}
which is
 again related to a Chi-subspace $\nu$ of 
the dual of the projective tensor  product  $H \hat \otimes_\pi H_1$
where $H_1$ is another Hilbert space, see Definition 
\ref{def:chi-weak-Dirichlet-process}. 
It is of course an extension of the notion of 
real-valued weak Dirichlet process, see Definition \ref{def:Dirweak}.
We illustrate that notion in the simple case when $H_1 = \R$,
$\nu = \nu_0\hat\otimes_\pi \mathbb{R} \equiv \nu_0$
  and $\nu_0$ is a Banach space continuously embedded in $H^*$: a
 process $\X$ is called  $\nu$-weak Dirichlet process if it is the sum
 of a local martingale $\M$ and a process $\A$ such that
 $[\A, N]_{\nu}=0$ for every real continuous local martingale $N$. 
This happens e.g. under the following assumptions.
\begin{itemize}
 \item[(i)] There is a family $(R(\epsilon), \epsilon > 0 )$ of non-negative random variables converging in probability, such
that $ Z(\epsilon) \le R(\epsilon), \epsilon > 0$, where
$$ Z(\epsilon) := \frac{1}{\epsilon} \int_0^T |\A(r+\epsilon) - \A(r)|_{\nu_0^*} |N(r+\epsilon) - N(r)| \ud r.$$
\item[(ii)] 
For all $h\in \nu_0$, $\lim_{\epsilon \to 0^+} \frac{1}{\epsilon}
 \int_0^t \,_{\nu_0}\left\langle \A(r+\epsilon) - \A(r), h\right\rangle_{\nu_0^*} 
(N(r+\epsilon) - N(r)) \ud r =0, \ \forall t \in [0,T].$
\end{itemize}
The  stochastic calculus theory developed in Sections \ref{SecStochInt} and \ref{SecTensor},
 allows to prove, in Theorem \ref{th:exlmIto}, a general It\^o's formula for 
convolution type processes and to show that they are 
 ${\chi}$-Dirichlet processes and a ${\nu}$-weak-Dirichlet processes; this is done in Corollary \ref{cor:X-barchi-Dirichlet}. 
The most important result
 is however Theorem \ref{th:prop6}. 
It generalizes to the Hilbert values framework, 
Proposition 3.10  of \cite{GozziRusso06} which states
that given $f:[0,T] \times \R \rightarrow \R$ of class
$C^{0,1}$ and $X$ is a weak Dirichlet process with
finite quadratic variation then
$Y(t) = f(t, X(t))$ is a real weak Dirichlet process. 
Our result is a Fukushima decomposition in the spirit of Dirichlet forms,
 which is the natural extension of Doob-Meyer
decomposition for semimartingales.
It can also be seen as a substitution-tool of It\^o's formula if $f$ is not smooth.
In particular, given some $H$-valued process $\X$, it allows 
to expand $f(t,\X(t))$
if $\X$ is a $\nu$-weak Dirichlet process which also has a $\chi$-quadratic variation.
In particular it fits the case when $\X$ is convolution type process.

 An important application consists in an application to stochastic control with state equation
given by the solution of an infinite dimensional stochastic evolution equation. This is the
object of \cite{FabbriRusso-preprint} where, denoted by $v$ a solution of the Hamilton-Jacobi-Bellman equation associated to
 the problem and by $\X$ a mild solution of the state equation,  the characterization of the process $t\mapsto v(t,\X(t))$ as 
a real weak Dirichlet process is used to prove a verification type theorem.

\bigskip

The scheme of the work is the following: in Section \ref{SecStochInt} we introduce the definition of forward integral  with values in Banach spaces and we discuss the relation with the Da Prato-Zabczyk and Young integrals in the Hilbert framework. Section \ref{SecTensor}, devoted to stochastic calculus, is the core of the paper: we introduce the concepts of  $\chi$-Dirichlet processes, $\nu$-weak-Dirichlet processes and we study their general properties. In Section \ref{sec:SPDEs}, the developed theory is applied to  the case of convolution type processes. 
In Appendix \ref{sec:preliminaries} we collect some useful results on projective tensor products of Hilbert spaces.
  

\section{Stochastic integrals}

\label{SecStochInt}

\subsection{Probability and stochastic processes}

In the whole paper we denote by $\left( \Omega,\mathscr{F},\mathbb{P}\right)$  a complete probability space and by $\left \{ \mathscr{F}_t \right \}_{t\geq 0}$ a filtration on $\left (\Omega , \mathscr{F}, \mathbb{P} \right )$ satisfying the usual conditions.
Given  $\tilde\Omega \in \mathscr{F}$ we denote  by  $I_{\tilde\Omega}\colon \Omega \to \{0,1\}$ the characteristic function of the set $\tilde\Omega$.
Conformally to the Appendix, given two (real) Banach spaces 
$B_1, B_2$ $\shl(B_1;B_2)$ will denote, as usual the space of linear bounded maps from $B_1$ to $B_2$.
Given a  Banach space $B$ we denote by $\mathscr{B}(B)$ the Borel $\sigma$-field on $B$. We fix $T>0$. 

By default we assume that all the processes $\X \colon [0,T]\times \Omega \to B$ are measurable functions with respect to the product $\sigma$-algebra $\mathscr{B}([0,T]) \otimes \mathscr{F}$ with values in $(B, \mathscr{B}(B))$. The dependence of a process on the variable $\omega\in\Omega$ is emphasized only if needed by the context. When we say that a process is 
continuous (respectively left continuous, right continuous, c\`adl\`ag, c\`agl\`ad ...) we mean that almost all its paths are continuous (respectively left-continuous, right-continuous, c\`adl\`ag, c\`agl\`ad...).

Let $\mathscr{G}$ be a sub-$\sigma$-field of $\mathscr{B}([0,T])\otimes \mathscr{F}$.
We say that such a process  $\X: ([0,T]\times \Omega,\mathscr{G}) \to B$ is 
measurable with respect to $ \mathscr{G}$ if it is the measurable in the usual sense.
It is said {\it strongly (Bochner) measurable} (with respect to $ \mathscr{G}$) if it
is the limit of $\mathscr{G}$-measurable countable-valued functions.  
We recall that if $\X$ is  measurable and $\X$  is c\`adl\`ag, c\`agl\`ad  or if $B$ is separable then $\X$ is strongly measurable. 
The $\sigma$-field $\mathscr{G}$  will not be mentioned when it is clearly designated. We denote  by  $\mathcal{P}$ the predictable $\sigma$-field on  $[0,T] \times \Omega$.
The processes $\X$ measurable on $(\Omega \times [0,T], {\mathcal P})$ are also called {\it predictable} processes. All those processes will be considered as strongly measurable, with respect to $\shp$.
Each time we use expressions as ``adapted'', 
''predictable'' etc... we will always mean ``with respect to the filtration $\left \{ \mathscr{F}_t \right \}_{t\geq 0}$''.
 
The \emph{blackboard bold} letters $\X$, $\Y$, $\M$... are used for Banach (or Hilbert)-space valued) processes, while notations $X$ (or $Y$, $M$...) are reserved for real valued processes. We also adopt the notations introduced in Appendix \ref{sec:preliminaries}.

\begin{Notation} \label{Not00}
We always assume the following convention: when needed all the Banach space  c\`adl\`ag processes (or functions) indexed by $[0,T]$ are extended setting $\X(t)=\X(0)$ for $t\leq 0$ and $\X(t)=\X(T)$ for $t\geq T$.
\end{Notation}

\begin{Definition}
\label{def:forward-integral}
Let $B_1$ and $B_2$ be two real separable Banach spaces. Let $\X\colon \Omega \times [0,T] \to \mathcal{L}(B_2,B_1)$ and $\Y\colon 
\Omega \times [0,T] \to B_2$ be two stochastic processes. Assume that $\Y$ 
is continuous and that $\mathbb P$-almost all
 trajectories of $\X$ are Bochner integrable.

If for almost every $t\in [0,T]$ the following limit  (in the norm of the space $B_1$) exists in probability
\begin{equation} \label{EForw}
\int_0^t \X(r) \ud ^- \Y(r):= \lim_{\epsilon \to 0^+} \int_0^t \X(r) 
\left (\frac{\Y(r+\epsilon) - \Y(r)}{\epsilon} \right ) \ud r
\end{equation}
and the random function $t\mapsto \int_0^t \X(r) \ud ^- \Y(r)$ admits a continuous (in $B_1$) version, we say that $\X$ is forward integrable with respect to $\Y$. That  version of $\int_0^\cdot \X(r) \ud ^- \Y(r)$ is called \emph{forward integral of $\X$ with respect to $\Y$}.
Replacing $\int_0^t \X(r) \ud ^- \Y(r)$ with $\int_0^t \X(r) \ud ^+ \Y(r)$ and
$\frac{\Y(r+\epsilon) - \Y(r)}{\epsilon}$ 
with $\frac{\Y(r) - \Y(r-\varepsilon)}{\epsilon}$
in \eqref{EForw},  $\int_0^\cdot \X(r) \ud ^+ \Y(r)$ is called \emph{backward integral of $\X$ with respect to $\Y$}.
\end{Definition}

The definition  above is a natural generalization of that given  in \cite{DiGirolamiRusso09} Definition 3.4; there the forward integral is a real valued process. 

\subsection{The semimartingale case}
Let  $H$ and $U$ be two separable Hilbert spaces; we adopt the notations introduced in Appendix \ref{sec:preliminaries}. An $U$-valued measurable process $\M \colon [0,T] \times \Omega \to U$  is called martingale if, for all $t\in [0,T]$, $\M$ is $\mathscr{F}_t$-adapted with $\mathbb{E} \left [ |\M(t)| \right ] <+\infty$ and $\mathbb{E}\left [ \M(t_2)  |\mathscr{F}_{t_1} \right ] = \M(t_1)$ for all $0\leq t_1 \leq t_2 \leq T$. The concept of (conditional) expectation for $B$-valued processes, for a  separable Banach space $B$, is recalled for instance in  \cite{DaPratoZabczyk92} Section 1.3. 
As in finite dimensions, local martingales are defined by usual techniques of localization. All the considered martingales and local martingales will be continuous.

We denote  by  $\mathcal{M}^2(0,T; H)$ the linear space of square integrable
 martingales equipped with the norm $|\M|_{\mathcal{M}^2(0,T; U)} := \left ( \mathbb{E} \sup_{t\in [0,T]} |\M(t)|^2 \right )^{1/2}$. It is a Banach space as stated in \cite{DaPratoZabczyk92}, Proposition 3.9.



Given a  local martingale $\M \colon [0,T] \times \Omega \to U$, the process $|\M|^2$ is a real local sub-martingale, see Theorem 2.11 in \cite{KrylovRozovskii07}. 
The increasing predictable process, vanishing at zero, appearing in the Doob-Meyer decomposition of $|\M|^2$
is denoted by $[\M]^{\mathbb{R}, cl}(t), t\in [0,T]$. It  is of course uniquely determined and continuous.

We recall some properties of the It\^o stochastic integral with respect to a local martingale $\M$. 
Call $\mathcal{I}_\M(0,T;H)$ the set of the processes $\X\colon [0,T]\times \Omega \to \mathcal{L}(U;H)$ that are 
 strongly measurable from $([0,T]\times \Omega, \mathcal{P})$ to $\mathcal{L}(U;H)$ and such that

\[
\|\X\|_{\mathcal{I}_\M(0,T;H)} := \left (\mathbb{E} \int_0^T \| \X(r) \|_{\mathcal{L}(U;H)}^2 \ud [\M]^{\mathbb{R}, cl}(r) \right )^{1/2} < +\infty.
\]
$\mathcal{I}_\M(0,T;H)$ endowed with the norm $\|\cdot\|_{\mathcal{I}_\M(0,T;H)}$ is a Banach space. The linear map
\[
\left \{
\begin{array}{l}
I\colon \mathcal{I}_\M(0,T;H) \to \mathcal{M}^2(0,T; H)\\
\X \mapsto \int_0^{\cdot} \X(r) \ud \M(r)
\end{array}
\right .
\]
is a contraction, see e.g. \cite{Metivier82} Section 20.4 (above Theorem 20.5).  As illustrated in  \cite{KrylovRozovskii07} Section 2.2 (above Theorem 2.14), the stochastic integral with respect to  $\M$ extends to the integrands $\X$ which are measurable from $([0,T]\times \Omega, \mathcal{P})$ to $\mathcal{L}(U;H)$ and such that
\begin{equation} \label{EChainRule}
\int_0^T \| \X(r) \|_{\mathcal{L}(U;H)}^2 \ud [\M]^{\mathbb{R}, cl}(r) < +\infty \qquad a.s.
\end{equation}

By $\mathcal{J}^2(0,T; U,H)$ we denote the family of integrands with respect to  $\M$ that satisfy (\ref{EChainRule}).

\begin{Theorem}
\label{th:integral-forrward=ito-martingale}
Let us consider a continuous local martingale $\M\colon [0,T]\times{\Omega} \to U$ and a c\`agl\`ad predictable $\mathcal{L}(U,H)$-valued process $\X$ satisfying (\ref{EChainRule}).
Then, the forward integral $\int_0^\cdot \X(r) \ud^- \M(r)$, defined in Definition \ref{def:forward-integral} exists and coincides with the It\^o integral $\int_0^\cdot \X(r) \ud \M(r)$.
\end{Theorem}

\begin{proof}[Proof of Theorem \ref{th:integral-forrward=ito-martingale}]
We follow the arguments related to the finite-dimensional case, see Theorem 2 of \cite{RussoVallois07}. Suppose first that $\X\in\mathcal{I}_\M(0,T;H)$. In this case it satisfies the hypotheses of the stochastic Fubini theorem in the form given in \cite{Leon90}. We have
\[
\int_0^t \X(r) \frac{\M(r+\epsilon) - \M(r)}{\epsilon} \ud r = \int_0^t \X(r) \frac{1}{\epsilon} 
\left ( \int_r^{r+\epsilon} \ud \M(\theta) \right) \ud r;
\]
applying the stochastic Fubini Theorem, the expression above is equal to
\[
\int_{0}^t \left ( \frac{1}{\epsilon}\int_{\theta-\epsilon}^{\theta} \X(r) \ud r  \right) \ud \M(\theta) + R_\epsilon(t),
\]
where $R_\epsilon(t)$ is a boundary term that converges to $0$ in probability, for any $t \in [0,T]$, so that we can ignore it. We can apply now the maximal inequality stated in \cite{Stein70}, Theorem 1: there exists a universal constant $C>0$ such that, for every $f\in L^2([0,t];\mathbb{R})$,
\begin{equation}
\label{eq:Este}
\int_0^t \left ( \sup_{\epsilon \in (0,1]} \left \vert \frac{1}{\epsilon} \int_{(r-\epsilon)
}^r f(\xi)  \ud \xi \right \vert \right )^2 \ud r \leq C \int_0^t f^2(r) \ud r.
\end{equation}
According to the vector valued version of the Lebesgue differentiation 
Theorem (see Theorem II.2.9 in \cite{DiestelUhl77}), the quantity $\frac{1}{\epsilon}\int_{(r-\epsilon)}^r \X(\xi) \ud \xi$ converges $\ud \mathbb{P} \otimes \ud r$ a.e. to $\X(r)$. Consequently (\ref{eq:Este}) and dominated convergence theorem imply
\begin{equation}
\label{eq:EIM}
\mathbb{E} \int_0^t \left \| \left ( \frac{1}{\epsilon}\int_{\theta-\epsilon}^{\theta} \X(r) \ud r \right ) -  \X(\xi) \right \|^2_{\mathcal{L}(U,H) } \ud [\M]^{\mathbb{R}, cl} (\xi) \; \xrightarrow{\epsilon\to 0} 0. 
\end{equation}
Since the integral is a contraction from $\mathcal{I}_\M(0,T;H)$ to $\mathcal{M}^2(0,T; H)$, this shows the claim for the processes in $\mathcal{I}_\M(0,T;H)$. If $\X$ is a c\`agl\`ad adapted process (and then a.s. bounded and therefore in ${J}^2(0,T; U,H)$) we use the same argument after localizing the problem by using the suitable sequence of stopping times defined by $\tau_n:= \inf \left \{ t \in [0,T] \; : \;  \int_0^t \left \| \X(r) \right \|^2_{\mathcal{L}(U,H)}  \ud [\M]^{\mathbb{R}, cl}(r)  \geq n  \right \}$ (and $+\infty$ if the set is void).
\end{proof}

An easier but still important statement  concerns the integration with respect bounded variation processes.
\begin{Proposition} \label{ItoBV}
 Let us consider a continuous bounded variation process $\V \colon [0,T]\times{\Omega} \to U$
and  let $\X$ be a c\`agl\`ad measurable  process  $[0,T] \times \Omega \rightarrow {\mathcal L}(U,H)$. 
Then the forward integral $\int_0^\cdot \X(r) \ud^- \V(r)$, defined in Definition \ref{def:forward-integral} exists and coincides with the Lebesgue-Bochner integral $\int_0^\cdot \X(r) \ud \V(r)$.
\end{Proposition}
\begin{proof} 
The proof is similar to the one of Theorem \ref{th:integral-forrward=ito-martingale}.
\end{proof}

As for the case of finite dimensional integrators, forward integrals also extend
Young type integrals. Let $\Y: [0,T] \times \Omega \rightarrow U$ be a $\gamma$-H\"older continuous
process and $\X$ be an $\shl(U;H)$-valued $\alpha$-H\"older continuous with $\alpha + \gamma > 1$.
Then, the so called {\it Young integral} $\int_0^t \X \ud^{y} \Y$ is well-defined, see Proposition \cite{gubinelli_lejay_tindel},
similarly as in the one-dimensional case. Moreover the $H$-valued integral process has 
also $\gamma$-H\"older continuous paths.

\medskip
An example of process $Y$ can arise as follows. Let $(e_n)$ ne an orthonormal basis in $U$,
$(\alpha_n)$ be a sequence of real numbers such that $\sum_n \alpha_n^2 < \infty$.
Let $\beta^n$ be sequence of real valued fractional Brownian motions of Hurst index
$H > \gamma$. It is not difficult to show that the $U$-valued random function
$\Y_t = \sum_{n=1} \alpha_n \beta_n e_n$ is well-defined.  In particular
for $0 \le t $, and $\ell \in \N$ we get
$$ E(\vert Y(t) - \Y(s) \vert^{2\ell}) \le \sum_{n=1} \alpha_n^2 (t-s)^{2\ell H}.$$
By Kolmogorov-Centsov theorem, taking $\ell$ large enough, it is possible to show 
that $\Y$ has has a $\gamma$-H\"older continuous version.

\medskip
\begin{Proposition} \label{P25} Let $\alpha, \gamma, \Y, \X$ as above.
Then the forward integral $\int_0^\cdot \X \ud^- \Y$ (resp. 
$\int_0^\cdot \X \ud^+ \Y$) 
 exists
and equals   $\int_0^\cdot \X \ud^y \Y$.
\end{Proposition}
\begin{proof} The proof is very similar to the one  in Section 2.2 of \cite{RussoVallois07}
concerning the one-dimensional case, by making use of Bochner integrals.
\end{proof}

A very natural research line is to study possible extensions of the results of this section to study the equivalence of the forward integral and 
the stochastic integral in UMD spaces introduced in \cite{vanNeervenVeraarWeis07}.
Similarly extensions of the results of Section \ref{sec:SPDEs} could include the study of the applicability of concepts introduced in Section \ref{SecTensor} to convolution  type processes in Banach space and, notably, to mild solution of stochastic evolution equations in UMD spaces  as developed e.g. in \cite{vanNeervenVeraarWeis08}. Some relevant results in that direction have been done by \cite{pronk}. 
Relevant applications to optimal control in the  framework of UMD spaces
were  obtained e.g. in \cite{BrzezniakSerrano13}.

\section{$\chi$-quadratic variation and $\chi$-Dirichlet processes}
\label{SecTensor}

\subsection{$\chi$-quadratic variation processes}
\label{sub4.2}

We denote by  $\mathcal{C}([0,T])$ the space of the real continuous processes equipped with the ucp (uniform convergence in probability) topology. Consider two real separable Banach spaces $B_1$ and $B_2$ with the same notations as in Appendix \ref{sec:preliminaries}. 

Following \cite{DiGirolamiRusso09, DiGirolamiRusso11} a {\it Chi-subspace} (of $(B_1\hat\otimes_\pi B_2)^*$) is defined as any Banach subspace $(\chi, |\cdot|_\chi)$ which is continuously embedded into $(B_1\hat\otimes_\pi B_2)^*$: in other words, there is some constant $C$ such that $|\cdot|_{(B_1\hat\otimes_\pi B_2)^*} \leq C |\cdot|_\chi$.

Concrete examples of Chi-subspaces are provided e.g. in \cite{DiGirolamiRusso11}, see Example 3.4. For example, for a fixed positive number $\tau$, in the case $B_1= B_2 = C([-\tau,0])$, the space of finite signed Borel measures on $[-\tau,0]^2$, $\mathcal{M}([-\tau, 0]^2)$ (equipped with the total variation norm) is shown to be a 
Chi-subspace and it is used, together with some specific subspaces, to prove
 It\^o-type formulas and Fukushima-type decompositions. Another concrete example will be used in Section \ref{sec:SPDEs} where, given
$B_1=B_2=H$ for some separable Hilbert $H$ space and denoted by $A$ the generator of a $C_0$-semigroup 
on $H$, we use the space $\bar\chi:= D(A^*)\hat\otimes_\pi D(A^*)$, see 
\eqref{eq:def-chi}. Among the several possible examples (see e.g. \cite{DaPratoZabczyk96}) we have for instance, in the case $H= 
L^2(\mathcal{O})$ for some bounded regular domain $\mathcal{O}\subseteq \mathbb{R}^m$, the heat semigroup with Dirichlet condition, in this case we have and $D(A^*)=D(A) = H^1_0(\mathcal{O}) \cap H^2(\mathcal{O})$.

Let $\chi$ be a generic Chi-subspace. We introduce the following definition.
\begin{Definition}
\label{def:covariation} 
Given $\X$  (respectively $\Y$)  a  $B_1$-valued (respectively $B_2$-valued) process, we say that $(\X, \Y)$ admits a $\chi$-covariation if the two following conditions are satisfied.
\begin{description}
\item[H1] For any sequence of positive real numbers $\epsilon_{n}\searrow 0$ 
there exists a subsequence $\epsilon_{n_{k}}$ such that 
\begin{equation} \label{FDefC}
\begin{split}
&\sup_{k}\int_{0}^{T} \frac{\left | (J\left(\X({r+\epsilon_{n_{k}}})-\X({r}))
\otimes(\Y({r+\epsilon_{n_{k}}})-\Y({r}))\right)\right |_{\chi^{\ast}} }{\epsilon_{n_{k}}} dr
\;< \infty\; a.s., 
\end{split}
\end{equation}
where $ J: B_{1}\hat{\otimes}_{\pi}B_{2} \longrightarrow (B_{1}\hat{\otimes}_{\pi}B_{2})^{\ast\ast}$ is the canonical injection between a space and its bidual.

\item[H2]
If we denote by $[\X,\Y]_\chi^{\epsilon}$ the application
\begin{equation}
\label{eq:def-chi-epsilon}
\left \{
\begin{array}{l}
[\X,\Y]_\chi^{\epsilon}:\chi\longrightarrow \mathcal{C}([0,T])\\[5pt]
\displaystyle
\phi \mapsto
\int_{0}^{\cdot} \tensor[_{\chi}]{\left\langle \phi,
\frac{J\left(  \left(\X({r+\epsilon})-\X({r})\right)\otimes \left(\Y({r+\epsilon})-\Y({r})\right)  \right)}{\epsilon} 
\right\rangle}{_{\chi^{\ast}}} dr, 
\end{array}
\right .
\end{equation}
the following two properties hold.
\begin{description}
\item{(i)} There exists an application, denoted by $[\X,\Y]_\chi$, defined on $\chi$ with values in $\mathcal{C}([0,T])$, 
satisfying
\begin{equation}
[\X,\Y]_\chi^{\epsilon}(\phi)\xrightarrow[\epsilon\longrightarrow 0_{+}]{ucp} [\X,\Y]_\chi(\phi), 
\end{equation} 
for every $\phi \in \chi\subset
(B_{1}\hat{\otimes}_{\pi}B_{2})^{\ast}$.
\item{(ii)} 
There exists a strongly measurable process
 $\widetilde{[\X,\Y]}_\chi:\Omega\times [0,T]\longrightarrow \chi^{\ast}$, 
such that
\begin{itemize}
\item for almost all 
$\omega \in \Omega$, $\widetilde{[\X,\Y]}_\chi(\omega,\cdot)$ is a (c\`adl\`ag) bounded variation process, 
\item 
$\widetilde{[\X,\Y]}_\chi(\cdot,t)(\phi)=[\X,\Y]_\chi(\phi)(\cdot,t)$ a.s. for all $\phi\in \chi$, $t\in [0,T]$.
\end{itemize}
\end{description}
\end{description}
\end{Definition}

\begin{Lemma} \label{RDGR}
In the setting of Definition \ref{def:covariation}, we set
\begin{equation} \label{Aepsilon}
A(\varepsilon) := \int_{0}^{T} \frac{\left | (J\left(\X({r+\epsilon})-\X({r}))
\otimes(\Y({r+\epsilon})-\Y({r}))\right)\right |_{\chi^{\ast}} }{\epsilon} dr.
\end{equation}
\begin{itemize}
\item[(i)]
If $\lim_{\epsilon \rightarrow 0} A(\epsilon) $ exists in probability
then Condition {\bf H1} of Definition \ref{def:covariation} is verified.
\item[(ii)] If $\lim_{\epsilon \rightarrow 0} A(\epsilon) = 0$ in probability
then $(\X, \Y)$  admits a $\chi$-covariation and 
$\widetilde{[\X,\Y]}$ vanishes. 
\end{itemize}
\end{Lemma}
\begin{proof}
It is an easy consequence of Remark 3.10 and Lemma 3.18 in \cite{DiGirolamiRusso11}.
\end{proof}

If $(\X,\Y)$ admits a $\chi$-covariation we call $\chi$-covariation of $(\X,\Y)$ the $\chi^{\ast}$-valued process  $\widetilde{[\X,\Y]}_\chi$ defined for every $\omega\in \Omega $ and $t\in [0,T]$ by 
$\phi \mapsto\widetilde{[\X,\Y]}_\chi(\omega,t)(\phi)=[\X,\Y]_\chi(\phi)(\omega,t)$. 
By abuse of notation, $[\X,\Y]_\chi$ will also be often called  $\chi$-covariation and it will be confused with $\widetilde{[\X,\Y]}_\chi$. We say that a process $\X$ admits a $\chi$-quadratic variation if $(\X, \X)$ admits a $\chi$-covariation. The process $\widetilde{[\X,\X]}_\chi$, often denoted by $\widetilde{[\X]}_\chi$, is also called $\chi$-quadratic variation of $\X$.

\begin{Definition}
\label{def:global-cov}
If $\chi = (B_1\hat\otimes_\pi B_2)^*$ the $\chi$-covariation is called 
\emph{global covariation}. In this case we  omit the index $(B_1\hat\otimes_\pi B_2)^*$ using the notations $[\X, \Y]$ and $\widetilde{[\X, \Y]}$.
\end{Definition}

\begin{Lemma}
\label{lm:RChiGlob}
Let $\X$ and $\Y$ as in Definition \ref{def:covariation}. The properties below hold.
\begin{itemize}
\item[(i)] If $(\X,\Y)$ admits a global covariation then it admits a $\chi$-covariation for any Chi-subspace $\chi$. Moreover $[\X, \Y]_{\chi}(\phi)= [\X, \Y](\phi)$ for all $\phi\in \chi$.
\item[(ii)] Suppose that  $\X$ and $\Y$ admit a scalar quadratic variation (Definitions \ref{def:per-nota}) and $(\X,\Y)$ has a tensor covariation (Definition \ref{def:tensorcovariation}),
 denoted   by  
$[\X, \Y]^{\otimes}$. 
Then $(\X,\Y)$  admits a global covariation $[\X, \Y]$.  In particular, recalling that $B_1\hat\otimes_\pi B_2$ is embedded in $(B_1\hat\otimes_\pi B_2)^{**}$, we have $\widetilde{[\X, \Y]} = [\X, \Y]^{\otimes}$.
\end{itemize}
\end{Lemma}
\begin{proof}
Part (i) follows by the definitions. For part (ii) the proof is a slight adaptation of the one of Proposition 3.14 in \cite{DiGirolamiRusso11}. In particular condition {\bf H1} holds using Cauchy-Schwarz inequality.
\end{proof}

The product of a real finite quadratic variation process 
and a zero real quadratic variation process is again a zero
quadratic variation processes. Under some conditions this
can be generalized to the infinite dimensional case as shown in the following proposition.
\begin{Proposition} \label{PFQVZ}
Let $i = 1,2$ and $\nu_i$ be a Banach space continuously
embedded in the dual $B_i^*$ of a real separable Banach space $B_i$.
Let consider the Chi-subspace of the type 
$\chi_1 = \nu_1 \hat\otimes_\pi B_2^* $ and   
$\chi_2 = B_1^* \hat\otimes_\pi \nu_2 $,
$\hat \chi_i =  \nu_i  \hat\otimes_\pi \nu_i$, $ i =1,2$. 
Let $\X$ (respectively $\Y$) be a process with values in $B_1$ (respectively $B_2$).
\begin{itemize} 
\item[(i)] Suppose that $\X$ admits a $\hat \chi_1$-quadratic variation 
and $\Y$  a zero scalar quadratic variation.
Then $[\X,\Y]_{\chi_1} = 0$. 
\item[(ii)] Similarly suppose that  $\Y$ admits a $\hat \chi_2$-quadratic variation 
and $\X$  a zero  scalar quadratic variation.
Then $[\X,\Y]_{\chi_2} = 0$. 
\end{itemize}
\end{Proposition}
\begin{proof}
We remark that Lemma \ref{lm:era52} implies that
$\chi_i$ and $\hat \chi_i$, $i = 1,2$ are indeed Chi-subspaces.
 By Lemma \ref{RDGR}(ii),
it is enough to show that $A(\varepsilon)$ defined in 
\eqref{Aepsilon} converge to zero, with $\chi = \chi_i, i = 1,2$.
By symmetry it is enough to show (i).

The Banach space $B_i$ it isometrically embedded in its bidual $B_i^{**},
i=1,2,$ so, since $\nu_1\subseteq B_1^*$ with continuous inclusion, we have $B_1 \subseteq B_1^{**} \subset \nu_1^*$ where the inclusion are continuous.

Moreover, since ${\chi_1} = \nu_1 \hat \otimes_\pi B_2^* \subseteq 
B_1^* \otimes_\pi B_2^* \subset (B_1 \hat \otimes_\pi B_2)^*$,
 with  continuous inclusions, taking into account Lemma \ref{lm:era52},
 we have $J(B_1 \hat \otimes_\pi B_2) \subset  
(B_1 \hat \otimes_\pi B_2)^{**} \subset {\chi_1}^*.$ 
Let $a \in B_1$ and  $b \in B_2$.
We have
\begin{multline}
\label{eq:step-da-richiamare}
\vert J(a \otimes b) \vert_{{\chi_1}^*} =  \sup_{\vert \varphi \vert_{\nu_1} \le 1, 
\vert \psi \vert_{B_2^*}  \le 1}  
\vert {}_{\chi_1^*} \langle J(a \otimes b), \varphi \otimes \psi 
\rangle_{\chi_1}\vert  \\
= \sup_{\vert \varphi \vert_{\nu_1} \le 1}
\vert {}_{\nu_1} \langle \varphi, a \rangle_{\nu_1^*} \vert
 \sup_{\vert \psi \vert_{B_2^{*}} \le 1}
\vert {}_{B_2^*} \langle \psi, b \rangle_{B_2^{**}} \vert 
= \vert a \vert_{\nu_1^*}\vert b \vert_{B_2^{**}} = \vert a \vert_{\nu_1^*}\vert b \vert_{B_2}.
\end{multline}

Consequently, 
with $a =  \X(r+\varepsilon) - \X(r)$ and $b =\Y(r+\varepsilon) - \Y(r)$
for $r \in [0,T]$, 
 we have 
\begin{multline}
 A(\varepsilon) =
 \int_{0}^{T} \frac{\left | (J\left(\X({r+\epsilon})-\X({r}))
\otimes(\Y({r+\epsilon})-\Y({r}))\right)\right |_{{\chi_1}^{\ast}} }{\epsilon} dr
=  \int_0^T \vert \X(r+\varepsilon) - \X(r) \vert_{\nu_1^*}
\vert \Y(r+\varepsilon) - \Y(r) \vert_{B_2} \frac{dr}{\varepsilon}\\ 
\le \left(\int_0^T \vert \X(r+\varepsilon) - \X(r) \vert_{\nu_1^*}^2 
\frac{dr}{\varepsilon}
\int_0^T  \vert \Y(r+\varepsilon) - \Y(r) \vert_{B_2}^2 
\frac{dr}{\varepsilon} \right)^{1/2} \\
= \left( \int_{0}^{T} \frac{\left | (J\left(\X({r+\epsilon})-\X({r}))
\otimes(\X({r+\epsilon})-\X({r}))\right)\right |_{\hat \chi_1^{\ast}} }{\epsilon} dr
\right)^{1/2}  \left( \int_{0}^{T} \frac{\left |\Y({r+\epsilon})-\Y({r}))
\right |^2_{B_2}}{\epsilon} dr
\right)^{1/2}.
\end{multline}
The last equality is obtained using an argument similar to (\ref{eq:step-da-richiamare}). The condition {\bf H1} related to the $\hat \chi_1$-quadratic variation of $\X$ and the zero scalar quadratic variation of  $\Y$, imply that previous expression converges to zero.
\end{proof}

\subsection{Tensor covariation and
classical tensor covariation}

The notions of tensor covariation recalled in Definition \ref{def:tensorcovariation}, denoted  by  $[\cdot, \cdot ]^{\otimes}$, concerns Banach space-valued processes. In the specific case when $H_1$ and $H_2$ are two separable Hilbert spaces and $\M \colon [0,T]\times\Omega \to H_1$, $\N \colon [0,T]\times\Omega \to H_2$ are two continuous local martingales, another (classical) notion of tensor covariation is defined, see for instance in Section 23.1 of \cite{Metivier82}. This is denoted by $[\M, \N ]^{cl}$ and (see Chapters 22 and 23 of \cite{Metivier82}) it is an $(H_1\hat\otimes_\pi H_2)$-valued process. Recall that $(H_1\hat\otimes_\pi H_2) \subseteq (H_1\hat\otimes_\pi H_2)^{**}$.

\begin{Remark}
\label{rm:propertyP}
We observe the following facts.
\begin{itemize}
\item[(i)] Taking into account Lemma \ref{lm:che-era-remark} we know that, given $h \in H_1$ and $k \in H_2$,  $h^* \otimes k^*$ can be considered as an element of $(H_1 \hat\otimes_\pi H_2)^*$. One has
\begin{equation}
\label{eq:propertyP}
[\M, \N ]^{cl}(t)(h^*\otimes k^*) = [ \left\langle \M,h \right\rangle, \left\langle \N,k \right\rangle ](t),
\end{equation}
where $h^*$ (respectively $k^*$) is associated with $h$ (respectively $k$) via Riesz theorem.  This property characterizes $[\M, \N ]^{cl}$, see e.g. \cite{DaPratoZabczyk92}, Section 3.4 after Proposition 3.11.
\item[(iii)] If $H_2=\mathbb{R}$ and $\mathbb{N}=N$ is a real continuous local martingale then, identifying $H_1\hat\otimes_\pi H_2$ with $H_1$, $[\M, \N ]^{cl}$ can be considered as  an  $H_1$-valued process. The characterization (\ref{eq:propertyP}) can be translated into
\begin{equation}
\label{eq:propertyP2}
[\M, \N ]^{cl}(t)(h^*) =
[ \left\langle \M,h \right\rangle, N ](t), \forall h \in H_1.
\end{equation}
By inspection, this allows us to see that
the classical covariation between $\M$ and $N$ can be expressed as
\begin{equation}
\label{eq:defNM}
[\M,N]^{cl}(t):= \M(t) N(t) -  \M(0) N(0) -
\int_0^t N(r) \ud \M(r) - \int_0^t \M(r) \ud N(r).
\end{equation}
\end{itemize}
\end{Remark}

In the sequel $H$ will denote a real separable Hilbert space.

\begin{Lemma}  
\label{lm:after-definition-chi-D}
Any continuous local martingale with values in $H$ has a scalar quadratic variation and a tensor quadratic variation. If $\M_1$ and $\M_2$ are continuous local martingales with values respectively in $H_1$ and $H_2$ then $(\M_1,\M_2)$ admits a tensor covariation.
\end{Lemma}
\begin{proof}
See Proposition 1.7 and Proposition 1.6 in \cite{DiGirolamiRusso11}.
\end{proof}

\begin{Lemma}
\label{lm:L411}
Let $\M$ (respectively $\N$) be a continuous local martingale with values in $H$.
Then $(\M,\N)$ admits a tensor covariation and 
\begin{equation}
\label{eq:eq413}
[\M, \N]^{\otimes}= [\M, \N]^{cl}. 
\end{equation}
In particular $(\M,\N)$ admits a global covariation 
and 
\begin{equation}
\label{eq:eq414}
\widetilde{[\M, \N]}= [\M, \N]^{cl}. 
\end{equation}
\end{Lemma}
\begin{proof} 
Thanks to Lemma \ref{lm:after-definition-chi-D},  $\M$ and $\N$ admit a  scalar quadratic variation and $(\M, \N)$ a tensor covariation. 
By Lemma \ref{lm:RChiGlob}(ii) they admit a global covariation. It is enough to show that they are equal as elements of $(H_1\hat\otimes_\pi H_2)^{**}$,
 so one needs to prove that
\begin{equation}
\label{eq:eq414bis}
[\M, \N]^{\otimes}(\phi)= [\M, \N]^{cl}(\phi),
\end{equation}
for every $\phi\in (H_1 \hat\otimes_\pi H_2)^*$. 

Given $h\in H_1$ and $k\in H_2$, we consider (via Lemma \ref{lm:che-era-remark}) $h^* \otimes k^*$ as an element of $(H_1\hat\otimes_\pi H_2)^*$.
According to Lemma \ref{lemmaTensor},
$H_1^*\hat\otimes_\pi H_2^*$ is sequentially dense in $(H_1\hat\otimes_\pi H_2)^*$
 in the weak-* topology.
 Therefore, taking into account item (i) of Remark \ref{rm:propertyP}
 we only need to show that
\begin{equation}
\label{eq:eq415}
[\M, \N]^{\otimes}(h^*\otimes k^*)= [\langle \M,h \rangle, \langle \N, 
k \rangle],
\end{equation}
for every $h \in H_1, k \in H_2$.
By the usual properties of Bochner integral the left-hand side of (\ref{eq:eq415}) is the limit of 
\begin{multline}
\frac{1}{\epsilon} \int_0^\cdot (M({r+\epsilon}) - M(r)) \otimes (N({r+\epsilon}) - N(r)) (h^*\otimes k^*) \ud r\\
= \frac{1}{\epsilon} \int_0^\cdot \langle (M({r+\epsilon}) - M(r)),h \rangle\langle (N({r+\epsilon}) - N(r)),k \rangle \ud r.
\end{multline}
Since $\left\langle \M, h \right\rangle$ and $\left\langle \N, k \right\rangle$ are real local martingales, the covariation $[\langle \M,h \rangle, \langle \N, k \rangle]$ exists and equals the classical covariation of local martingales because of Proposition 2.4(3) of \cite{RussoVallois93-Oslo}.
\end{proof}

\begin{Lemma}
\label{lm:Prop3}
Let $\M\colon [0,T] \times \Omega \to H$ be a continuous local martingale and $\Z$ a  measurable process from $([0,T]\times \Omega, \mathcal{P})$ to $H^*$ and such that 
$\int_0^T \| \Z(r) \|^2 \ud [\M]^{\mathbb{R}, cl}(r) < +\infty$ $a.s$. 
We define
\begin{equation}
\label{eq:defX-perlemma}
X(t) := \int_0^t \left\langle \Z(r) , \ud \M(r)\right\rangle.
\end{equation}
Then $X$ is a real continuous local martingale and, for every continuous real local martingale $N$, the (classical, one-dimensional) covariation process $[X,N]$ is given by
\begin{equation}
\label{eq:[XN]-perlemma}
[X,N](t) = \int_0^t \left\langle \Z(r) , \ud [\M,N]^{cl}(r)\right\rangle;
\end{equation}
in particular the integral in the right-side is well-defined.
\end{Lemma}
\begin{proof}
The fact that $X$ is a local martingale is part of the result of Theorem 2.14 in \cite{KrylovRozovskii07}. For the other claim we can reduce, using a sequence of suitable stopping times as in the proof of Theorem \ref{th:integral-forrward=ito-martingale}, to the case
 in which $\Z$, $\M$ and $N$ are square integrable martingales. Taking into account the characterization (\ref{eq:propertyP2}) and the discussion developed in  \cite{Meyer77}, page 456, (\ref{eq:[XN]-perlemma}) follows.
\end{proof}

When one of the processes is real the formalism of global covariation can be simplified as shown in the following proposition.
\begin{Proposition}
\label{pr:1}
Let be $\X\colon [0,T] \times \Omega \to H$ a Bochner integrable process and $Y\colon [0,T] \times \Omega \to \mathbb{R}$ a real valued process. Suppose 
the following.
\begin{itemize}
\item[(a)] For any $\epsilon$,
$\frac{1}{\epsilon} \int_0^T |\X(r+\epsilon) - \X(r)|
 |Y(r+\epsilon) - Y(r)| \ud r$ is bounded by a r.v. $A(\epsilon)$
such that  $A(\epsilon)$ converges in probability when $\epsilon \to 0$.
\item[(b)] For every $h\in H$ the  limit 
\[
C(t)(h):= \lim_{\epsilon \to 0^+} \frac{1}{\epsilon} \int_0^t \left \langle h, \X(r+ \epsilon) - \X(r) \right\rangle (Y(r+ \epsilon) - Y(r)) \ud r
\]
exists ucp and there exists a continuous process $\tilde C\colon [0,T] \times \Omega \to H$ such that $\left\langle \tilde C(t,\omega), h \right\rangle = C(t)(h)(\omega)$ for $\mathbb{P}$-a.s. $\omega\in\Omega$, for all $t\in [0,T]$ and $h\in H$.
\end{itemize}
Then $\X$ and $Y$ admit a global covariation and, after identifying $H$ with $(H \hat\otimes_{\pi}\mathbb{R})^*$, $\tilde C= \widetilde{[\X,Y]}$.
\end{Proposition}
\begin{proof}
Taking into account the identification of $H$ with $(H\hat\otimes_\pi \mathbb{R})^*$  the result is a consequence of Corollary 3.26 of \cite{DiGirolamiRusso11}. 
\end{proof}

\begin{Proposition}
\label{pr:thA}
If $\M\colon [0,T] \times \Omega \to H$ and $N\colon [0,T] \times \Omega \to \mathbb{R}$ are continuous local martingales. Then
$\M$ and $N$ admit a global covariation and $\widetilde{[\M,N]}=[\M,N]^{cl}$.
\end{Proposition}
\begin{proof}
We have to check the conditions stated in Proposition \ref{pr:1} 
 for $\tilde C$ equal to the right side of (\ref{eq:defNM}). Concerning (a),
 by Cauchy-Schwarz inequality we have
\[
\frac{1}{\epsilon} \int_0^T |N(r+\epsilon) - N(r)| |\M(r+\epsilon) - \M(r)| \ud r
 \leq [N,N]^{\epsilon, \mathbb{R}}(T) [\M,\M]^{\epsilon, \mathbb{R}}(T).
\]
Since both $N$ and $\M$ are local martingales they admit a scalar quadratic variation (as recalled in Lemma \ref{lm:after-definition-chi-D}), the result is established. Concerning (b), taking into account (\ref{eq:propertyP2}), we need to prove that for any $h\in H$
\begin{equation}
\label{eq:FIF}
\lim_{\epsilon \to 0} \frac{1}{\epsilon} \int_0^\cdot ( M^h(r+\epsilon) - M^h(r)) (N(r+\epsilon) - N(r)) \ud r = [\left\langle M, h \right\rangle, N]
\end{equation}
ucp, where $M^h$ is the real local martingale $\langle \M, h \rangle$. (\ref{eq:FIF}) follows by Proposition 2.4(3) of \cite{RussoVallois93-Oslo}.
\end{proof}

\subsection{$\chi$-Dirichlet and $\nu$-weak Dirichlet processes}

We have now at our disposal all the elements we need to introduce the concepts of $\chi$-Dirichlet process and $\nu$-weak Dirichlet process.
\begin{Definition}
\label{def:chi-Dirichlet-process}
Let $\chi\subseteq (H\hat\otimes_{\pi} H)^*$ be a Chi-subspace.
A continuous $H$-valued process $\X \colon ([0,T] \times \Omega,\mathcal{P}) \to H$ is called \emph{$\chi$-Dirichlet process} if there exists a decomposition $\X = \M +\A$ where
\begin{itemize}
 \item[(i)] $\M$ is a continuous local martingale,
 \item[(ii)] $\A$ is a continuous $\chi$-zero quadratic variation process with $\A(0)=0$.
\end{itemize}
\end{Definition}

\begin{Definition} \label{DNuDir}
Let $H$ and $H_1$ be two separable Hilbert spaces.
Let $\nu\subseteq (H\hat\otimes_{\pi} H_1)^*$ be a Chi-subspace. 
A continuous adapted $H$-valued process $\A \colon [0,T] \times \Omega \to H$ is said to be \emph{$\mathscr{F}_t$-$\nu$-martingale-orthogonal} if  $[ \A, \N ]_\nu=0$, for any $H_1$-valued continuous local martingale $\N$.
\end{Definition}

As we have done for the expressions ``stopping time'', ''adapted'', ``predictable''... since we  always use the 
filtration $\mathscr{F}_t$, we simply write 
\emph{$\nu$-martingale-orthogonal} instead of $\mathscr{F}_t$-$\nu$-martingale-orthogonal.

\begin{Definition}
\label{def:chi-weak-Dirichlet-process}
Let $H$ and $H_1$ be two separable Hilbert spaces.
Let $\nu\subseteq (H\hat\otimes_{\pi} H_1)^*$ be a Chi-subspace. 
A continuous $H$-valued process $\X \colon [0,T] \times \Omega \to H$ is called \emph{$\nu$-weak-Dirichlet process} if it
is adapted and there exists a decomposition $\X = \M +\A$ where
\begin{itemize}
 \item[(i)] $\M$ is  an  $H$-valued continuous local martingale,
 \item[(ii)] $\A$ is an $\nu$-martingale-orthogonal process with $\A(0)=0$.
\end{itemize}
\end{Definition}
The decomposition of a real weak Dirichlet process is unique, see Remark 3.5 of \cite{GozziRusso06}.  For the infinite dimensional case we have the following result.
\begin{Proposition} \label{PDenseUnique}
Let $\nu\subseteq (H\hat\otimes_{\pi} H)^*$ be a Chi-subspace. Suppose that
 $\nu$ is dense in $(H\hat\otimes_\pi H)^*$. Then any decomposition of
 a $\nu$-weak-Dirichlet process $\X$ is unique.
\end{Proposition}
\begin{proof}
Assume that $\X=\M^1+\A^1 = \M^2 + \A^2$ are two decompositions where $\M^1$ 
and $\M^2$ are continuous local martingales and $\A^1, \A^2$ are $\nu$-martingale-orthogonal processes. If we call $\M:= \M^1- \M^2$ and $\A:= \A^1-\A^2$ we have $0=\M+\A$. 

By Lemma \ref{lm:L411},  $\M$ has a global quadratic variation. 
In particular it also has a $\nu$-quadratic variation and,
 thanks to the bilinearity  of the $\nu$-covariation,
\[
0 = [\M,0]_\nu = [\M,\M+\A]_\nu = [\M,\M]_\nu +  [\M,\A]_\nu = [\M,\M]_\nu + 0 =[\M,\M]_\nu.
\]
We prove now that $\M$ has also zero global quadratic variation. 
We have denoted  by  $\mathcal{C}([0,T])$ the space of the real continuous processes defined on $[0,T]$. We introduce, for $\epsilon >0$, the operators
\begin{equation}
\label{eq:peremailaFrancesco}
\left \{
\begin{array}{l}
\displaystyle
[\M,\M]^\epsilon \colon  (H\hat\otimes_\pi H)^* \to \mathcal{C}([0,T])\\[5pt]
\displaystyle
([\M,\M]^\epsilon (\phi))(t) := \frac{1}{\epsilon} \int_0^t \tensor[_{(H\hat\otimes_\pi H)}]{\left\langle (\M_{r+\epsilon} - \M_r)^{\otimes 2}, \phi \right\rangle}{_{(H\hat\otimes_\pi H)^*}} \ud r.
\end{array}
\right .
\end{equation}
Observe the following.
\begin{itemize}
\item[(a)] $[\M,\M]^\epsilon$ are linear and bounded operators.
\item[(b)] For $\phi\in (H\hat\otimes_\pi H)^*$ the limit $[\M,\M](\phi):= \lim_{\epsilon\to 0} [\M,\M]^\epsilon (\phi)$ exists.
\item[(c)] If $\phi\in \nu$ we have $[\M,\M](\phi)=0$.
\end{itemize}
Thanks to (a), (b) and Banach-Steinhaus theorem for $F$-spaces (see Theorem 17, Chapter II in \cite{DunfordSchwartz58}) we know that $[\M,\M]$ is linear and bounded. Thanks to (c) and the fact that the inclusion 
$\nu \subseteq  (H\hat\otimes_\pi H)^*$ is dense, it follows $[\M,\M]=0$. 
By Lemma \ref{lm:L411} $[\M, \M]$ coincides with the classical quadratic variation $[\M, \M]^{cl}$ and it is characterized by
\[
0 = [\M, \M]^{cl} (h^*,k^*) = [\langle \M , h \rangle, \langle \M , k \rangle],
\]
by Remark 
\ref{rm:propertyP}(i).
Since $\M(0)=0$ and therefore $\langle \M , h \rangle(0)=0$ it follows that $\langle \M , h \rangle\equiv 0$ for any $h\in H$. Finally $\M\equiv 0$,
which concludes the proof.
\end{proof}

\begin{Proposition}
\label{pr:dirichlet-weak-dirichlet}
Let $H$ and $H_1$ be two separable Hilbert spaces.
Let  $\chi = \chi_0 \hat\otimes_\pi \chi_0$
for some $\chi_0$ Banach space continuously embedded in $H^*$.
Define $\nu = \chi_0\hat\otimes_\pi H_1^*$.
 Then an $H$-valued continuous zero $\chi$-quadratic variation process
 $\A$ is a $\nu$-martingale-orthogonal process. 
\end{Proposition}

\begin{proof}
Taking into account Lemma \ref{lm:era52},
$\chi$ is  a Chi-subspace of $(H\hat\otimes_\pi H)^*$ and
 $\nu$ is a Chi-subspace of $(H\hat\otimes_\pi H_1)^*$.
Let $\N$ be a continuous local martingale with values in $H_1$. 
We need to show that $[\A, \N]_\nu=0$. 
We consider the random maps
$T^\epsilon \colon  \nu \times \Omega \to \mathcal{C}([0,T])$ defined by
\[
T^\epsilon(\phi) := [\A, \N]_\nu^{\epsilon} (\phi) = \frac{1}{\epsilon} 
\int_0^\cdot \!\!\! \,_{\nu^*}\!\!\left \langle (\A(r+\epsilon) - \A(r)) 
\otimes (\N(r+\epsilon) - \N(r)), \phi \right\rangle_{\nu} \! \ud r,
\]
for $\phi \in \nu$.

\emph{Step 1.} Suppose that $\phi = h^*\otimes k^*$ for $h^* \in\chi_0$ and $k \in H_1$. Then 
\begin{multline}
T^\epsilon(\phi)(t) = \frac{1}{\epsilon} \int_0^t \,_{\chi_0^*}\left \langle (\A(r+\epsilon) - \A(r)), h^* \right \rangle_{\chi_0} \, \left\langle (\N(r+\epsilon) - \N(r)), k \right\rangle_{H_1} \ud r\\
\leq 
\left [ \frac{1}{\epsilon} \int_0^t \,_{\chi_0^*}\left \langle (\A(r+\epsilon) - \A(r)), h^* \right \rangle^2_{\chi_0} \ud r 
  \frac{1}{\epsilon} \int_0^t \left\langle (\N(r+\epsilon) - \N(r)), k \right\rangle^2_{H_1} \ud r\right ]^{1/2}\\
=
\left [\frac{1}{\epsilon} \int_0^t \,_{\chi^*}\left \langle (\A(r+\epsilon) - \A(r))^{\otimes 2}, h^*\otimes h^* \right \rangle_{\chi} \ud r \right ]^{\frac{1}{2}} \\
 \times \left [\frac{1}{\epsilon} \int_0^t \left\langle (\N(r+\epsilon) - \N(r)), k \right\rangle^2_{H_1} \ud r\right  ]^{\frac{1}{2}},
\end{multline}
that converges ucp to
\[
\left ( [\A,\A](t)(h^* \otimes h^*) [\N,\N]^{\rm cl}(t)(k^*\otimes k^*)
 \right )^{1/2} = 0,
\]
since the quadratic quadratic variation of a local martingale
is the classical one and taking into account item (i) of Remark \ref{rm:propertyP}.

\emph{Step 2.} We denote  by  $\mathcal{D}$ the linear combinations of elements of the form
 $h^* \otimes k^* $ for $h^* \in \chi_0$ and $k \in H_1$. We remark that
 $\mathcal{D}$ is dense in $\nu$. From the convergence found in \emph{Step 1}, it follows that, for every $\phi\in \mathcal{D}$, ucp we have $T^{\epsilon} (\phi) \xrightarrow{\epsilon \to 0} 0$.

\emph{Step 3.} We consider a generic $\phi\in\nu$. By Lemma \ref{lm:che-era-remark},
for $t \in [0,T]$ it follows
\begin{multline}
\label{eq:dacitare-primo-passo}
|T^{\epsilon}(\phi)(t)| \leq |\phi|_{\nu} \int_0^t \frac{\left | (\A(r+\epsilon) - \A(r)) \otimes (\N(r+\epsilon) - \N(r)) \right |_{\nu^*}} {\epsilon} \ud r \\
= |\phi|_{\nu} \frac{1}{\epsilon}\int_0^t \left | (\N(r+\epsilon) - \N(r)) \right |_{H_1} \left| (\A(r+\epsilon) - \A(r)) \right |_{\chi_0^*} \ud r\\
\leq 
|\phi|_{\nu} \left ( \frac{1}{\epsilon}\int_0^t \left | (\N(r+\epsilon) - \N(r)) \right |^2_{H_1} \ud r \; 
\frac{1}{\epsilon}\int_0^t
\left| (\A(r+\epsilon) - \A(r)) \right |^2_{\chi_0^*} \ud r \right )^{\frac12}\\
=
|\phi|_{\nu} \Bigg ( \frac{1}{\epsilon}\int_0^t \left | (\N(r+\epsilon) - \N(r)) \right |^2_{H_1} \ud r  \times 
\frac{1}{\epsilon}\int_0^t
\left| (\A(r+\epsilon) - \A(r))^{\otimes 2} \right |_{\chi^*} \ud r \Bigg )^{\frac12}.
\end{multline}
To prove that $[\A, \N]_\nu=0$ we check the corresponding conditions 
\textbf{H1} and \textbf{H2} of Definition \ref{def:covariation}.
 By Lemma \ref{lm:L411} we know that $\N$ admits a global 
quadratic variation i.e. a $(H_1\otimes H_1)^*$-quadratic variation.
 By condition \textbf{H1} of the Definition \ref{def:covariation}  related to
 $(H_1\otimes H_1)^*$-quadratic variation for the process 
$\N$ and the $\chi$-quadratic variation of $\A$, for any sequence $(\epsilon_n)$
converging to zero, there is a subsequence   $(\epsilon_{n_k})$
such that  the sequence $T^{\epsilon_{n_k}}(\phi)$ is bounded  for any $\phi$ 
in the $\mathcal{C}[0,T]$ metric a.s.
 condition \textbf{H1} of the $\nu$-covariation. By Banach-Steinhaus for $F$-spaces (Theorem 17, Chapter II in \cite{DunfordSchwartz58}) it follows that $T^\epsilon (\phi)\xrightarrow{\epsilon\to 0} 0$ ucp for all $\phi \in \nu$ and so condition \textbf{H2} and the final result follows.
\end{proof}

\begin{Corollary}
\label{Cor:codi}
Assume that the hypotheses of Proposition \ref{pr:dirichlet-weak-dirichlet} 
are satisfied. If $\X$ is a $\chi$-Dirichlet process then we have the following.
\begin{itemize}
 \item[(i)] $\X$ is a $\nu$-weak-Dirichlet process.
 \item[(ii)] $\X$ is a $\chi$-weak Dirichlet process.
\item[(iii)] $\X$ is a $\chi$-finite-quadratic-variation process.
\end{itemize}
\end{Corollary}
\begin{proof}
(i)   follows by Proposition \ref{pr:dirichlet-weak-dirichlet}. As far as (ii) is concerned, let   $\X = \M + \A$ be 
a  $\chi$-Dirichlet process decomposition, where $\M$ is a local martingale. Setting $H_1 = H$, then $\chi$ is included
in $\nu$, so  Proposition   \ref{pr:dirichlet-weak-dirichlet}
implies that $\A$ is a $\chi$-orthogonal process 
and so (ii) follows. We prove now (iii):  By Lemma \ref{lm:L411} and Lemma \ref{lm:RChiGlob}(i) $\M$ admits a $\chi$-quadratic variation.  By
 the bilinearity of the $\chi$-covariation, it is enough to show
that $[\M,\A]_\chi = 0$. This follows from item (ii).
\end{proof}

\begin{Proposition}
\label{pr:DiGirolamiRusso37}
Let $B_1$ and $B_2$ be two real separable Banach spaces and $\chi$ a Chi-subspace of $(B_1\hat\otimes_{\pi} B_2)^*$. Let $\X$ and $\Y$ be two stochastic processes with values respectively in $B_1$ and $B_2$ such that $(\X,\Y)$ admits a
 $\chi$-covariation. Let $G$ be a continuous measurable process $G: [0,T] \times \Omega \to \mathcal{K}$ where $\mathcal{K}$ is a closed separable subspace of $\chi$. 
Then for every $t\in [0,T]$
\begin{equation}   		\label{eq SDFR}  
\int_{0}^{t} \tensor[_\chi]{\langle G(\cdot,r), [\X,\Y]^{\epsilon} (\cdot, r)\rangle}{_{\chi^*}} \ud r 
\xrightarrow[\epsilon \longrightarrow 0]{} 
\int_{0}^{t} \tensor[_\chi]{\langle G(\cdot,r),d\widetilde{[\X,\Y]}(\cdot,r)\rangle}{_{\chi^*}}
\end{equation}
in probability.
\end{Proposition}
\begin{proof}
See \cite{DiGirolamiRusso11Fukushima} Proposition 3.7.
\end{proof}

We state below the most important result related to
 the stochastic calculus part of the paper. It generalizes the 
finite dimensional result contained in \cite{GozziRusso06} Theorem 4.14. Indeed it provides a (generalized) 
Fukushima-Dirichlet decomposition of a $C^{0,1}$ function $u$ of  a (suitably defined weak) Dirichlet process $\X$.
The result below consitutes indeed a chain rule type allowing to expand 
a non-smooth function $u$ of 
a convolution type process $\X$, in substitution of It\^o formula.
Similar results have been used for several purposes in stochastic control and forward-backward stochastic differential equations
(FBSDEs): in the first case in finite dimension by \cite{GozziRusso06Stoch}  to obtain a verification theorem and 
by \cite{FuhrmanTessitore} to obtain identify the solution the process $Z$ as part of the solution
of an infinite dimensional FBSDE, as 
 a  {\it generalized gradient} of a solution a semilinear PDE.
The difficulty here is that the underlying process $\X$ is far from being a semimartingale.

 The definition of real weak Dirichlet process is recalled in Definition \ref{def:Dirweak}.

\begin{Theorem}
\label{th:prop6}
Let $\nu_0$ be a Banach subspace continuously embedded in $H^*$. Define $\nu:= \nu_0\hat\otimes_\pi\mathbb{R}$ and $\chi:=\nu_0\hat\otimes_\pi\nu_0$. Let
 $F\colon [0,T] \times H \to \mathbb{R}$  be a $C^{0,1}$-function.  Denote with $\partial_x F$ the Frechet derivative of $F$ with respect to  $x$ and assume that the mapping $(t,x) \mapsto \partial_xF(t,x)$ is continuous from $[0,T]\times H$ to $\nu_0$.  Let $ \X(t) = \M(t) + \A(t)$ for $t\in [0,T]$ be an $\nu$-weak-Dirichlet process with finite $\chi$-quadratic variation. Then $Y(t):= F(t, \X(t))$ is a (real) weak Dirichlet process with local martingale part
\[
R(t) = F(0, \X(0)) + \int_0^t \left\langle \partial_xF(r,\X(r)), \ud \M(r) \right\rangle.
\]
\end{Theorem}
\begin{proof}
By definition $\X$ can be written as the sum of a continuous local martingale $\M$ and a $\nu$-martingale-orthogonal process $\A$. 

Let $N$ be a real-valued local martingale. Taking into account 
Lemma \ref{lm:Prop3} and that the covariation of two  real local martingales
defined in \eqref{DefRealIntCov},
coincides with the classical covariation,
 it is enough to prove that 
\[
[F(\cdot, \X(\cdot)), N](t) = \int_0^t \left\langle \partial_xF(r,\X(r)), \ud [\M,N]^{cl}(r) \right\rangle, \qquad \text{for all } t\in [0,T].
\] Let $t \in [0,T]$. We evaluate the $\epsilon$-approximation of
 the covariation, i.e.
\[
\frac{1}{\epsilon} \int_0^t \left ( F(r+\epsilon, \X(r+\epsilon)) - F(r, \X(r)) \right ) 
\left ( N(r+\epsilon) - N(r) \right ) \ud r.
\]
It equals $I_1(t,\epsilon) + I_2(t,\epsilon)$, where
\[
I_1(t,\epsilon) =  \int_0^t \left ( F(r+\epsilon, \X(r+\epsilon)) - F(r+\epsilon, \X(r)) \right ) 
 \frac{ \left ( N(r+\epsilon) - N(r) \right )}{\epsilon} \ud r
\]
and
\[
I_2(t,\epsilon) =  \int_0^t \left ( F(r+\epsilon, \X(r)) - F(r, \X(r)) \right ) 
\frac{\left( N(r+\epsilon) - N(r) \right )}{\epsilon} \ud r.
\]
We prove now that
\begin{equation}
\label{eq:11}
I_1(t,\epsilon) \xrightarrow{\epsilon\to 0} \int_0^t \left\langle \partial_xF(r, \X(r)) , \ud [\M,N]^{cl}(r) \right\rangle
\end{equation}
in probability; in fact $I_1(t,\epsilon) = I_{11}(t,\epsilon) + I_{12}(t,\epsilon)$ where 
\[
I_{11}(t,\epsilon) := \int_0^t \frac{1}{\epsilon} \left\langle \partial_xF(r, \X(r)), \X(r+\epsilon) - \X(r) \right\rangle (N(r+\epsilon) - N(r)) \ud r,
\]
\begin{multline}
\nonumber
I_{12}(t,\epsilon) := \int_0^1 \int_0^t \frac{1}{\epsilon} \left\langle \partial_xF(r+\epsilon, a\X(r) + (1-a) \X(r+\epsilon)) - \partial_xF(r, \X(r)), \right .\\ 
\left . \X(r+\epsilon) - \X(r) \right\rangle (N(r+\epsilon) - N(r)) \ud r \ud a.
\end{multline}
Now we apply Proposition \ref{pr:DiGirolamiRusso37} with $B_1= H$, $B_2 = \mathbb{R}$, $\X=\M$, $\Y=N$, $\chi=\nu$ 
\begin{equation}
\label{eq:EQ11bis}
I_{11}(t,\epsilon) \xrightarrow{\epsilon\to 0} \int_0^t \tensor[_{\nu}]{\left\langle \partial_xF(r, \X(r)) , \ud \widetilde{[\X,N]}(r) \right\rangle}{_{\nu^*}}.
\end{equation}
Recalling that  $\X=\M+\A$, we remark that
 $[\X,N]_\nu$ exists and the $\nu^*$-valued process $\widetilde{[\X,N]}_\nu$
 equals $\widetilde{[\M,N]}_\nu + \widetilde{[\A,N]}_\nu = \widetilde{[\M,N]}_\nu$,
 since $\A$ is a $\nu$-martingale orthogonal process. 
Taking into account  the formalism of Proposition \ref{pr:1},
Lemma \ref{lm:RChiGlob}(i) and Proposition \ref{pr:thA} if $\Phi \in H \equiv H^{*}$, we have
\[
\tensor[_{\nu}] {\left\langle \Phi, \widetilde{[\M,N]_\nu} 
\right\rangle}{_{\nu^*}}
=  \tensor[_{\nu_0}]{ \left\langle \Phi, \widetilde{[\M,N]_\nu}
 \right\rangle}{_{\nu_0^*}} 
 =
 \tensor[_{H^*}]{ \left\langle \Phi, \widetilde{[\M,N]}
 \right\rangle}{_{H^{**}}} 
=
 \tensor[_{H^*}]{ \left\langle \Phi, [\M,N]^{cl}
 \right\rangle}{_{H}}.
\]
Consequently, it is not difficult to show that
the right-hand side of (\ref{eq:EQ11bis}) gives $\int_0^t {}_{H^*}\left \langle \partial_x F(r, \X(r)) , \ud [\M,N]^{cl}(r) \right\rangle_{H}$.

For a fixed $\omega\in\Omega$ we consider the function $\partial_xF$ restricted to $[0,T]\times K$ where $K$ is the (compact) subset of $H$ obtained as convex hull of $\{ a \X(r_1) + (1-a) \X(r_2) \; : \; r_1, r_2 \in [0,T] \}$. $\partial_xF$ restricted to $[0,T]\times K$ is uniformly continuous with values in $\nu_0$.
Consequently, for  $\omega$-a.s. 
\begin{equation}
\label{eq:per-f-2}
|I_{12} (t,\epsilon)| \leq \int_0^T \delta\left (\partial_xF_{|[0,T]\times K} ; \epsilon +
\sup_{|r-t|\leq \epsilon} |\X(r) - \X(t)|_{\nu_0^*} \right )
\times  | \X(r+\epsilon) - \X(r)|_{\nu_0^*}
\frac{1}{\epsilon}|N(r+\epsilon) - N(r)| \ud r,
\end{equation}
where, for a uniformly continuous function 
$g:[0,T] \times K \rightarrow \nu_0$, $\delta(g ; \epsilon)$ is 
the modulus of continuity $\delta(g; \epsilon):= \sup_{|x-y|\leq \epsilon}
 |g(x) - g(y)|_{\nu_0}$. 
In previous formula we have identified $H$ with $H^{**}$ so that
$\vert x \vert_H \le \vert x \vert_{\nu_0^*}, \forall x \in H$.
So 
\eqref{eq:per-f-2} is lower than
\begin{multline} \label{E54}
 \delta\left (\partial_xF_{|[0,T]\times K} ; \epsilon +
\sup_{|s-t|\leq \epsilon} |\X(s) - \X(t)|_{\nu_0^*} \right )
\times \left ( \int_0^T \frac{1}{\epsilon}|N(r+\epsilon) - N(r)|^2 \ud r 
\int_0^T \frac{1}{\epsilon}|( \X(r+\epsilon) - \X(r) )|_{\nu_0^*}^2 \ud r \right )^{1/2}\\
=\delta\left (\partial_xF_{|[0,T]\times K} ; \epsilon +
\sup_{|s-t|\leq \epsilon} |\X(s) - \X(t)|_{\nu_0^*} \right )
\times \left ( \int_0^T \frac{1}{\epsilon}|N(r+\epsilon) - N(r)|^2 \ud r 
\int_0^T \frac{1}{\epsilon}|( \X(r+\epsilon) - \X(r) ) ^{\otimes 2}|_{\chi^*} \ud r \right )^{1/2},
\end{multline}
when $\varepsilon \rightarrow 0$, where we have used Lemma 
\ref{lm:che-era-remark} 
with the usual identification.
The right-hand side of \eqref{E54}, of course converges to zero, 
 since $\X$ (respectively $N$) is a $\chi$-finite quadratic variation process (respectively a real finite quadratic variation process)
and $X$ is also continuous as a $\nu_0^*$-valued process.

To conclude the proof of the proposition we only need to show that $I_2(t,\epsilon)\xrightarrow{\epsilon\to 0}0$. This is relatively simple since 
\[
I_2(t,\epsilon) = \frac{1}{\epsilon} \int_0^t \Gamma(u,\epsilon) \ud 
N(u) + R(t,\epsilon),
\]
where $R(t, \epsilon)$ is a boundary term such that $R(t, \epsilon)\xrightarrow{\epsilon\to 0}0$ in probability and
\[
\Gamma(u,\epsilon) = \frac{1}{\epsilon} \int_{(u-\epsilon)_+}^{u} 
\left ( F(r+\epsilon, \X(r)) - F(r, \X(r)) \right ) \ud r.
\]
Since $\int_0^T (\Gamma(u,\epsilon))^2 \ud [N](u) \to 0$ in probability, Problem 2.27, chapter 3 of \cite{KaratzasShreve88} implies that $I_2(\cdot, \epsilon) \to 0$ ucp. The result finally follows.
\end{proof}


\section{The case of convolution type processes }



\label{sec:SPDEs}

This section concerns applications of the stochastic calculus via regularization to convolution type processes. As a particular case the results apply to mild solutions of infinite dimensional stochastic evolution equations.

Assume that $H$ and $U$ are real separable Hilbert spaces, $Q \in \mathcal{L}(U)$, $U_0:=Q^{1/2} (U)$. Assume that  $\W_Q=\{\W_Q(t):0\leq t\leq T\}$  is an $U$-valued $\mathscr{F}_t$-$Q$-Wiener process with $\W_Q(0)=0$, $\mathbb{P}$ a.s. See Sections 2.1 and 2.2 of \cite{GawareckiMandrekar10} for the definition and properties of $Q$-Wiener processes and the definition of stochastic integral with respect to $\W_Q$. Denote by  $\mathcal{L}_2(U_0, H)$ the Hilbert space of the Hilbert-Schmidt operators from $U_0$ to $H$\footnote{The definition and the first properties of Hilbert-Schmidt operators can be found in Appendix A.2 of \cite{PeszatZabczyk07} and for more details in \cite{Ryan02}.}. 

In the whole section, as usual in the literature we refer (see e.g. \cite{DaPratoZabczyk92, GawareckiMandrekar10}) we will always identity, via Riesz Theorem, the spaces $H$ and $H^*$. In this way the duality between $H^*$ and $H$ reduces simply to the scalar product in $H$.

\begin{Lemma}
\label{lm:perPChainRule}
Consider an $\mathcal{L}(U,H)$-valued predictable process $\A$ such that $\int_0^T {\rm Tr} [\A(r) Q^{1/2} (\A(r) Q^{1/2})^*]  \ud r < \infty$ a.s. and define 
\begin{equation} \label{AA}
\M _t = \int_0^t  \A (r) \ud \W_Q(r), t \in [0,T].
\end{equation}
If $\X$ is an $H$-valued predictable process such that 
\begin{equation} \label{EChainRule11}
\int_0^T \langle \X (r), \A(r) Q^{1/2} (\A  Q^{1/2})^* \X(r) \rangle \ud r  < \infty, \,\, {\rm a.s.},
\end{equation}
then
\begin{equation} \label{AA1}
N (t) = \int_0^t \langle \X (r), \ud \M(r) \rangle, t \in [0,T],
\end{equation}
is well-defined  and it equals 
$N (t) = \int_0^t \langle \X (r), \A(r) \ud \W_Q(r) \rangle$ for $t \in [0,T]$.
\end{Lemma}
\begin{proof}
See Section 4.7 of   \cite{DaPratoZabczyk92}.
\end{proof}

We recall in the following proposition some significant properties of the stochastic integral with respect to local martingales.
\begin{Proposition} \label{PChainRule}
Let $\A$, $\M$ and $N$ as in Lemma \ref{lm:perPChainRule}. 
\begin{enumerate}
\item[(i)] $\N(t) := \int_0^t \X(r) d\M(r)$ is a well-defined $(\mathscr{F}_t)$-local martingale.
\item[(ii)] Let $\K$ be an   $(\mathscr{F}_t)$-predictable process such that $\K \X$ fulfills \eqref{EChainRule}. Then the It\^o-type stochastic integral  $\int_0^t \K d \N$ for $t \in [0,T]$ is well-defined and it equals $\int_0^t \K \X d \M$. 
\item[(iii)] If $\M$ is a $Q$-Wiener process $W_Q$, then, whenever $\X$ is such that 
\begin{equation} \label{EAAA}
\int_0^T Tr \left [ \left ( \X(r) Q^{1/2} \right ) \left (\X(r) Q^{1/2} \right)^* \right ] \ud r < +\infty \ {\rm a.s.},
\end{equation}
then $\N(t) = \int_0^t \X(r) d\W_Q(r)$ is a local martingale
and $[N]^{\R, cl}(t) = \int_0^t \left ( \X(r) Q^{1/2} \right ) \left (\X(r) Q^{1/2} \right)^* \ud r$.
\item[(iv)] If in (iii), the expectation of the
quantity \eqref{EAAA} is finite, 
then $\N(t) = \int_0^t \X(r) d\W_Q(r)$ is a square integrable continuous martingale. 
\item[(v)] If $\M$ is defined as in \eqref{AA} and $\X$ 
fulfills \eqref{EChainRule11}, then $\M$ is a real local martingale.
If moreover,  the expectation of \eqref{EChainRule11} is finite, then 
$N$, defined in (\ref{AA1}), is a square integrable martingale.
\end{enumerate}
\end{Proposition}
\begin{proof}
For (i) see \cite{KrylovRozovskii07} Theorem 2.14 page 14-15. For (ii) see \cite{MetivierPellaumail80}, proof of Proposition 2.2 Section 2.4. (iii) and (iv) are contained in \cite{DaPratoZabczyk92} Theorem 4.12 Section 4.4. (v) is a consequence of (iii) and (iv) and of Lemma \ref{lm:perPChainRule}.
\end{proof}

We recall the following fact that concerns the classical tensor covariation.
\begin{Lemma}
\label{lm:lemmaDPZ-covariation}
Let $\Psi\colon ([0,T] \times \Omega, \mathcal{P}) \to \mathcal{L}_2(U_0,H)$ be a strongly measurable process satisfying 
\[
\int_0^T \left \| \Psi(r) \right \|^2_{\mathcal{L}_2(U_0,H)} \ud r < +\infty \qquad a.s. 
\]
Consider the local martingale $\M(t):= \int_0^t \Psi(r) \ud \W_Q(r)$. Then $[\M, \M]^{cl}(t) = \int_0^t g(r) \ud r$, where $g(r)$ is the element of $H\hat\otimes_\pi H$ associated with the nuclear operator $G_g(r):=\Big ( \Psi(r) Q ^{1/2} \Big ) \Big ( \Psi(r) Q^{1/2} \Big )^*$.
\end{Lemma}
\begin{proof}
See \cite{DaPratoZabczyk92} Section 4.7.
\end{proof}

We introduce now the class of convolution type processes that includes mild solutions of (\ref{eq:SPDEIntro}). We denote   by  $A\colon D(A) \subseteq H \to H$ the generator of the $C_0$-semigroup $e^{tA}$ (for $t\geq 0$) on $H$. The reader may consult   for instance \cite{BDDM07} Part II, Chapter 1 
for basic properties of $C_0$-semigroups.
$A^*$ denotes the adjoint of $A$.
 $D(A)$ and $D(A^*)$ are Hilbert spaces 
 when endowed with the graph norm: for any $x\in D(A^*)$ (respectively $x\in D(A)$), 
$|x|_{D(A^*)}^2 := |x|^2 + |A^*x|^2$ (respectively $|x|_{D(A)}^2 := |x|^2 + |Ax|^2$).

Observe that, as a consequence of the identification of $H$ and $H^*$, $A^*$ is the generator of a $C_0$-semigroup on $H$ and $D(A^*)$ is continuously embedded in $H$. 

Let $b$ be a predictable process with values in $H$ and $\sigma$  be a predictable process  with values in $\mathcal{L}_2(U_0, H)$ such that
\begin{equation}
\label{eq:hp-per-existence-mild}
\mathbb{P} \left [ \int_0^T |b(t)| + \|\sigma(t)\|^2_{\mathcal{L}_2(U_0, H)} \ud t < +\infty \right ] =1.
\end{equation}
We are interested in the process
\begin{equation}
\label{eq:SPDE-mild}
\X(t) = e^{tA} x  + \int_0^t e^{(t-r)A} b(r) \ud r + \int_0^t e^{(t-r)A} \sigma(r) \ud \W_Q(r).
\end{equation}
In this paper, a process of this form is called \emph{convolution type process}. 
We define 
\begin{equation}
\label{eq:defY}
\Y(t):= \X(t) - \int_0^t b(r) \ud r - \int_0^t \sigma (r) \ud \W_Q(r) -x.
\end{equation}

\begin{Lemma}
\label{lm:conOndrejat}
Let $b$  (respectively $\sigma$) be a 
predictable process with values 
 in $H$  
(respectively with values  in $\mathcal{L}_2(U_0, H)$) such that 
(\ref{eq:hp-per-existence-mild}) is satisfied. Let $\X(t)$ be defined by (\ref{eq:SPDE-mild}) and 
$\Y$ defined by (\ref{eq:defY}). If $z\in D(A^*)$ we have $\left\langle \Y(t) , z \right\rangle = \int_0^t \left\langle \X(r), A^*z \right\rangle \ud r$.
\end{Lemma}
\begin{proof}
See \cite{Ondrejat04} Theorem 12.
\end{proof}

We want now to prove that $\Y$ has zero-$\chi$-quadratic variation for a suitable space $\chi$. We will see that the space
\begin{equation}
\label{eq:def-chi}
\bar\chi:= D(A^*)\hat\otimes_\pi D(A^*)
\end{equation}
does the job. We set $\bar\nu_0:= D(A^*)$ which is clearly continuously embedded into $H$. 
By Lemma \ref{lm:era52}, $\bar \chi$ is a Chi-subspace of 
 $(H\hat\otimes_\pi H)^*$.

\begin{Proposition}
\label{pr:Y-zero-chi-quad-var}
The process $\Y$ has zero $\bar\chi$-quadratic variation.
\end{Proposition}
\begin{proof}
Observe that, thanks to Lemma 3.18 in \cite{DiGirolamiRusso11} it is enough
 to show that 
\[
I(\epsilon):= \frac{1}{\epsilon}\int_0^T |(\Y(r+\epsilon) - \Y(r))^{\otimes 2}|_{\bar\chi^*}
 \ud r \xrightarrow{\epsilon \to 0}0, \qquad \text{in probability}.
\]
In fact, identifying $\bar\chi^*$ with ${\mathcal Bi}(\bar\nu_0, \bar\nu_0 ; \mathbb{R})$
(see after Notation \ref{Not1})
we get
\begin{multline}
I(\epsilon) = \frac{1}{\epsilon} \int_0^T \sup_{|\phi|_{\bar{\nu}_0},\,
 |\psi|_{\bar{\nu}_0} \leq 1} \left | \left\langle (\Y(r+\epsilon) - \Y(r)), \phi \right \rangle \left\langle (\Y(r+\epsilon) - \Y(r)), \psi \right \rangle \right | \ud r \\
\leq 
\frac{1}{\epsilon} \int_0^T \sup_{|\phi|_{\bar{\nu_0}}, \, |\psi|_{\bar{\nu}_0} \leq 1} \left \{ \left | \int_r^{r+\epsilon} \left\langle (\X(\xi), A^*\phi \right \rangle \ud \xi \right |  
\left | \int_r^{r+\epsilon} \left\langle (\X(\xi), A^*\psi \right \rangle \ud \xi \right | \right \} \ud r,
\end{multline}
where we have used Lemma \ref{lm:conOndrejat}. This  is smaller than $\frac{1}{\epsilon} \int_0^T  \left | \int_r^{r+\epsilon} |\X(\xi)| \ud \xi \right |^2 \ud r
\leq \epsilon \sup_{\xi\in [0,T]} |\X(\xi)|^2$, which converges to zero almost surely.
\end{proof}

\begin{Corollary}
\label{cor:X-barchi-Dirichlet}
The process $\X$ is a $\bar \chi$-Dirichlet process.
 Moreover it
 is also a $\bar\chi$ finite quadratic variation process
 and a $\bar\nu_0\hat\otimes_\pi \mathbb{R}$-weak-Dirichlet process.
\end{Corollary}
\begin{proof}
For $ t \in [0,T]$, we have $\X(t) = \M(t) + \A(t)$, where $\M(t) := x + \int_0^t \sigma (r) \ud \W_Q(r)$, $\A(t) := \V(t) + \Y(t)$ and $\V(t) = \int_0^t b(r) dr$. $\M$ is a local martingale by Proposition \ref{PChainRule} (iii) and $\V$ is a bounded variation process.  By  Proposition \ref{PFQVZ} and Lemma \ref{lm:RChiGlob}(i), we get $[\V,\V]_{\bar{\chi}} =  [\V,\Y]_{\bar{\chi}} = [\Y,\V]_{\bar{\chi}} = 0.$  By Proposition \ref{pr:Y-zero-chi-quad-var} and the bilinearity of  the $\bar{\chi}$-covariation, it yields that $\A$ has a zero $\bar{\chi}$-quadratic variation and so $\X$ is a $\bar{\chi}$-Dirichlet process. The second part of the statement is a consequence of Corollary  \ref{Cor:codi}.
\end{proof}
\begin{Remark} \label{RDirichlet}
Corollary \ref{cor:X-barchi-Dirichlet} allows to apply
 the Fukushima type Theorem \ref{th:prop6}
to  expand a process of the form $F(t,\X(t))$, where $F$ is a $ C^{0,1}$-function and  $\X$  is a convolution type process, so for instance the 
solution of a  stochastic differential equation of evolution type. This looks as a It\^o generalized chain rule
for non-smooth function $F$ and it was applied  in \cite{FabbriRusso-preprint} to obtain the stochastic control 
verification Theorem  6.11.
\end{Remark}


In the sequel we will denote by $UC([0,T]\times H; D(A^*))$ ($D(A^*)$ is equipped with the graph norm) the $F$-space of the functions $G: [0,T] \times H \rightarrow  D(A^*)$ 
 which are uniformly continuous on sets $[0,T]\times D$ for any closed ball $D$ of $H$,
 equipped with the topology of the uniform 
 convergence on closed balls.

\smallskip

The theorem below  generalizes for some aspects the It\^o formula
of \cite{DiGirolamiRusso11}, i.e. their Theorem 5.2, 
to the case when the second derivatives do not necessarily 
belong to the Chi-subspace $\chi$.
\begin{Theorem}
\label{th:exlmIto}
Let $F\colon [0,T] \times H \to \mathbb{R}$ of class $C^{1,2}$. Suppose that $(t,x)\mapsto \partial_xF(t,x)$ belongs to $ UC([0,T]\times H; D(A^*))$. Let $\X$ be an  $H$-valued process process
admitting a $\bar\chi$-quadratic variation.
We suppose the following.
\begin{itemize}
\item[(i)] There exists a (c\`adl\`ag) bounded variation process $C\colon [0,T] \times \Omega \to (H\hat\otimes_\pi H)$ such that, for all $t$ in $[0,T]$ and $\phi\in \bar\chi$, 
\[
C(t,\cdot)(\phi) = [\X, \X]_{\bar\chi}(\phi)(t, \cdot) \qquad a.s.
\]
\item[(ii)] For every continuous function $\Gamma\colon [0,T] \times H \to D(A^*)$ the following integral exists:
\begin{equation} \label{ForwI}
\int_0^t \left\langle \Gamma(r,\X(r)), \ud ^- \X(r) \right\rangle.
\end{equation}
\end{itemize}
Then, for any $t\in [0,T]$,
\begin{multline}
\label{eq:Ito}
F(t,\X(t)) = F(0, \X(0)) + \int_0^t \left\langle \partial_rF(r, \X(r)), \ud^- \X(r) \right\rangle \\
+\frac12 \int_0^t  {}_{(H \hat\otimes_{\pi} H)^*}  
\left \langle \partial_{xx}^2 F(r, \X(r)) , \ud C(r) 
\right\rangle_{H \hat\otimes_{\pi} H} 
  + \int_0^t \partial_r F(r, \X(r)) \ud r, \quad a.s.
\end{multline}
\end{Theorem}

\begin{proof}
\emph{Step 1.} Let $\{e_i^*\}_{i\in \mathbb{N}}$ be an orthonormal basis of $H$ made of elements of $D(A^*)\subseteq H$. 
This is always possible since $D(A^*)\subseteq H$ densely embedded, via a Gram-Schmidt orthogonalization procedure. For $N\geq 1$ we denote by $P_N\colon H \to H$ the orthogonal projection on the span of the vectors 
$\{e_1^*, .. , e_N^*\}$. $P_\infty\colon H \to H$ simply denotes the identity.

Let us for a moment omit the time  dependence on $F$,
which is supposed to be of class $C^2$ from $H$ to $\R$.
We define $F_N\colon H \to \mathbb{R}$ as $F_N(x):= F(P_N (x))$. We have
\begin{equation} \label{57bis}
\partial_xF_N(x) = P_N \partial_xF(P_N (x))
\end{equation}
and 
\[
\partial_{xx}^2F_N(x) = (P_N\otimes P_N) \partial_{xx}^2 F (P_N (x)),
\]
where the latter  equality has to be understood as
\begin{multline}
\label{eq:3}
\,_{(H\hat\otimes_{\pi} H)^*}\! \left\langle \partial_{xx}^2 F_N (x) , h_1\otimes h_2 \right\rangle_{(H\hat\otimes_{\pi} H)} = 
\,_{(H\hat\otimes_{\pi} H)^*}\!\left\langle \partial_{xx}^2 F (P_N(x)) , (P_N(h_1))\otimes (P_N(h_2)) \right\rangle_{H\hat\otimes_{\pi} H},
\end{multline}
for all $h_1, h_2 \in H$. $\partial_{xx}^2 F_N (x)$ is an element of $(H\hat\otimes_{\pi} H)^*$ but it belongs to $(D(A^*)\hat\otimes_{\pi} D(A^*))$ as well; indeed it can be written as
\[
\sum_{i,j=1}^N \,_{(H\hat\otimes_{\pi} H)^*}\!\left\langle \partial_{xx}^2F(P_N (x)), e_i^*\otimes e_j^* \right\rangle_{H\hat\otimes_{\pi} H} \, \left ( e_i^* \otimes  e_j^* \right )
\]
and $e_i^*\otimes e_j^*$ are in fact elements of 
$D(A^*)\hat\otimes_{\pi} D(A^*)$.

We come back now again to the time dependence notation $F(t,x)$.
We can apply the It\^o formula proved in \cite{DiGirolamiRusso11}, 
Theorem 5.2, and with the help of Assumption (i), we find
\begin{multline}
\label{eq:Ito-N}
F_N(t,\X(t)) = F_N(0, \X(0)) + \int_0^t \left\langle \partial_xF_N(r, \X(r)), \ud^- \X(r) 
\right\rangle \\
+\frac12 \int_0^t   {}_{(H \hat\otimes_{\pi} H)^*} \left\langle \partial_{xx}^2 F_N(r, \X(r)),
 \ud C(r) \right\rangle_{H \hat\otimes_{\pi} H}  + \int_0^t \partial_r F_N(r, \X(r)) \ud r.
\end{multline}

\emph{Step 2.} We consider, for fixed $\epsilon >0$, the map
\[
\left \{
\begin{array}{l}
T_\epsilon \colon UC([0,T]\times H; D(A^*)) \to L^0(\Omega)\\[5pt]
T_\epsilon \colon  \displaystyle G \mapsto \int_0^t \left\langle G(r, \X(r)), \frac{\X(r+\epsilon) - \X(r)}{\epsilon} \right\rangle \ud r,
\end{array}
\right .
\]
where the set    $L^0(\Omega)$ of all real random variables
 is equipped with the topology of the convergence in probability. 
Assumption (ii) implies that $\lim_{\epsilon\to 0} T_\epsilon G$ exists for every $G$.
By Banach-Steinhaus for $F$-spaces (see Theorem 17, Chapter II in \cite{DunfordSchwartz58}) it follows that the map
\[
\left \{
\begin{array}{l}
UC([0,T]\times H; D(A^*)) \to L^0(\Omega)\\[5pt]
\displaystyle
G \mapsto \int_0^t \left\langle G(r, \X(r)), \ud ^-\X(r) \right\rangle
\end{array}
\right .
\]
is linear and continuous.

\emph{Step 3.} If $K\subseteq H$ is a compact set then the set $P(K) := \left \{ P_N(y) \; : \; y\in K, \; N\in  \mathbb{N}\cup {+\infty} \right \}$ is compact as well. To illustrate this fact we show that for any sequence of elements belonging to $P(K)$ we can extract a subsequence which converges to an element of $P(K)$. A generic sequence of elements in $P(K)$ has the form $\left \{ P_{N_l} (y_l) \right \}_{l \in \mathbb{N}}$ where, for any $l$, $N_l \in  \mathbb{N}\cup
 \{+\infty\}$ and $y_l \in K$. Since $K$ is compact we can assume,  without loss of generality, that $y_l$ converges, for ${l\to + \infty}$, to some $y\in K$. If $\{ N_l \}$ assumes only a finite number of values then (passing if necessary to a subsequence) $N_l \equiv \bar N$ for some 
$\bar N \in \mathbb{N}\cup \{+\infty\}$ and then $P_{N_l} (y_l) \xrightarrow{l\to+\infty} P_{\bar N} (y)$ that belongs of course to $P(K)$. Otherwise we can assume (passing if necessary to a subsequence) that $N_l \xrightarrow{l\to+\infty} +\infty$ and then 
it is not difficult to prove that $P_{N_l} (y_l) \xrightarrow{l\to+\infty} y$, which belongs to $P(K)$ since
$y = P_\infty y$. In particular, being $\partial_xF$ continuous,
\[
\mathcal{D}:= \left \{ \partial_xF (P_N (x)) \; : \; x\in K, \; 
N\in  \mathbb{N}\cup \{+\infty\} \right \}
\]
is compact in $D(A^*)$. Since the sequence of maps $\{ P_N \}$ is uniformly continuous, it follows that
\begin{equation}
\label{eq:unif-conv-1}
\sup_{x \in \mathcal{D}} \vert (P_N - I) (x) \vert  \xrightarrow{N\to\infty} 0. 
\end{equation}

\emph{Step 4.} We show now that 
\begin{equation}
\label{eq:first-term}
\lim_{N\to\infty} \int_0^t \left\langle \partial_xF_N(r, \X(r)), \ud^- \X(r) \right\rangle =
 \int_0^t \left\langle \partial_xF(r, \X(r)), \ud^- \X(r) \right\rangle
\end{equation}
holds in probability for every $t\in [0,T]$.
Let $K$ be a compact subset of $H$. In fact 
\[
\sup_{t \in [0,T], x\in K} |\partial_xF(t, P_N (x)) - \partial_xF(t,x)| 
 \xrightarrow{N\to\infty}  0,
\]
since $\partial_xF$ is continuous. On the other hand
\[
\sup_{t \in [0,T], x\in K} |(P_N - I)(\partial_xF(t, P_N x))|  
\xrightarrow{N\to\infty}  0,
\]
because of (\ref{eq:unif-conv-1}). Consequently, by \eqref{57bis}, 
\begin{equation} \label{eq:second}
\partial_x F_N  \to \partial_x F, 
\end{equation}
uniformly on each compact, with values in $H$.
This yields that 
$\omega$-a.s. 
\[
\partial_x F_N (r, \X (r)) \to \partial_x F (r, \X (r)),
\]
uniformly on each compact. By Step 2, then (\ref{eq:first-term}) follows.

%

\emph{Step 5.} Finally, we prove that
\begin{equation}
\label{eq:second-term}
\lim_{N\to\infty} \frac12 \int_0^t  {}_{(H \hat\otimes_{\pi} H)^*} \left \langle \partial_{xx}^2 F_N(r, \X(r)), \ud C(r) \right\rangle_{H \hat\otimes_{\pi} H} 
 = \frac12 \int_0^t  {}_{(H \hat\otimes_{\pi} H)^*} \left\langle \partial_{xx}^2 F(r, \X(r)) , \ud C(r) \right\rangle_{H \hat\otimes_{\pi} H}.
\end{equation}
For a fixed $\omega \in\Omega$ we define $K(\omega)$ the compact set as $K(\omega):= \left \{ \X(t)(\omega) \; : \; t\in[0,T] \right \}$. We write
\begin{multline}
\left |  \int_0^t  {}_{(H \hat\otimes_{\pi} H)^*}\left\langle \partial_{xx}^2F_N(r,\X(x)) - \partial_{xx}^2 F (r,\X(x)) , \ud C(r)
 \right\rangle_{H \hat\otimes_{\pi} H} \right | (\omega)\\
\leq
\sup_{\substack { y\in K(w)\\ t\in [0,T]}} \left | \partial_{xx}^2F_N(t,y) -
 \partial_{xx}^2F(t,y)  \right |_{(H\hat\otimes_{\pi} H)^*} \int_0^T \ud \vert C\vert_{\rm var}(r) (\omega),
\end{multline}
where $r \mapsto \vert C(r) \vert_{\rm var}$ is to total variation fonction
of $C$. Using arguments similar to those used in proving (\ref{eq:second}) one can see that $\partial^2_{xx} F_N \xrightarrow{N\to\infty} \partial^2_{xx} F$ uniformly on each compact. Consequently 
\[
\sup_{r\in [0,T]} | (\partial^2_{xx} F_N -
 \partial^2_{xx} F) (r, \X(r)) |_{(H \hat\otimes_\pi H)^*} \xrightarrow{N\to\infty} 0.
\]
Since $C$ has bounded variation, finally (\ref{eq:second-term}) holds.


\emph{Step 6.} Since $F_N$ (respectively $\partial_r F_N$) converges uniformly on each compact to 
$F$ (respectively $\partial_r F$), when $N\to\infty$, then 
$$ \int_0^t \partial_r F_N(r, \X(r)) \ud r   \xrightarrow{N\to\infty}  
\int_0^t \partial_r F (r, \X(r)) \ud r.$$
Taking the limit when $N\to\infty$ in
 (\ref{eq:Ito-N}), finally provides (\ref{eq:Ito}).
\end{proof}

\begin{Lemma}
\label{lm:per-lemmaIto}
The conditions (i) and (ii) of Theorem \ref{th:exlmIto} are verified if $\X = \M + \V + \S $, where  $\M$ is a local martingale, $\V$ is an $H$-valued bounded variation process, and  $\S$ is a process verifying
\[
\left\langle \S, h \right\rangle (t) = \int_0^t \left\langle \Z (r) ,
 A^* h \right\rangle \ud r, \qquad \text{for all } h\in D(A^*),
\]
for some  measurable process $\Z$ with $\int_0^T |\Z(r)|^2 \ud r <+\infty$ a.s. 
\end{Lemma}
\begin{proof}
By Lemma \ref{lm:L411}, $\M$ admits a global quadratic quadratic variation which can be identified with $[\M, \M]^{cl}$. On the other hand $\A = \V + \S$ has a zero $\bar \chi$-quadratic variation, by Proposition \ref{PFQVZ} and the  bilinearity character of the  $\bar \chi$-covariation.
$\X$  is therefore a $\bar \chi$-Dirichlet process. By Corollary \ref{Cor:codi} and again the bilinearity of the $\bar \chi$-covariation, we obtain that
$\X$ has a finite $\bar \chi$-quadratic variation. Taking also into account Lemma \ref{lm:L411} and Lemma \ref{lm:RChiGlob}(i), we get $\widetilde{[\X,\X]}_{\bar \chi}(\Phi)(\cdot) =  \langle [\M, \M]^{cl}, \Phi \rangle$ if $\Phi \in \bar \chi$. Consequently, we can set $C =  [\M, \M]^{cl}$ and condition (i)  is verified.

To prove (ii) consider a continuous function $\Gamma\colon [0,T] \times H \to D(A^*)$. The integral of   $(\Gamma(r,\X(r)))$ with respect to  the semimartingale  $\M + \V$ where $\M(t) = x + \int_0^t \sigma (r) \ud W_Q(r) $  exists and  equals the classical It\^o integral, by  Proposition \ref{th:integral-forrward=ito-martingale} and Proposition \ref{ItoBV}.  Therefore, we only have to prove that
\[
\int_0^t \left\langle \Gamma(r,\X(r)), \ud^- \S (r) \right\rangle,
 \qquad t\in [0,T],
\]
exists. For every $t\in [0,T]$ the $\epsilon$-approximation of such an integral gives, up to a remainder boundary term $C(\epsilon,t)$ which converges in probability to zero,
\begin{multline}
\label{eq:era38}
\frac{1}{\epsilon}\int_0^t \left\langle \Gamma(r,\X(r)), \S(r+\epsilon) - 
\S(r) \right\rangle \ud r 
=  \frac{1}{\epsilon}\int_0^t \int_{r}^{r+\epsilon} \left\langle \Z(u), A^*\Gamma(r,\X(r)) \right\rangle \ud u \ud r \\
=  \frac{1}{\epsilon}\int_0^t \int_{u-\epsilon}^{u} \left\langle \Z(u), A^*\Gamma(r,\X(r)) \right\rangle \ud r \ud u \xrightarrow{\epsilon \to 0} \int_0^t \left\langle \Z(u), A^*\Gamma(u,\X(u)) \right\rangle \ud u,
\end{multline}
in probability by classical Lebesgue integration theory. The right-hand side of (\ref{eq:era38}) has obviously a continuous modification so \eqref{ForwI} exists by definition and condition (ii) is fulfilled. 
\end{proof}

Next result can be considered a It\^o formula for mild type processes. An interesting  contribution
in this direction, but in a different spirit appears in \cite{RoecknerMild}.
\begin{Corollary}
\label{lm:Ito-pre}
Assume that $b$ is a 
predictable process with values in $H$ and $\sigma$ 
is
a predictable process  with values in $\mathcal{L}_2(U_0, H)$ satisfying
 (\ref{eq:hp-per-existence-mild}). Define $\X$ as in (\ref{eq:SPDE-mild}). Let $x$ be an element of $H$. Assume that $f\in C^{1,2}([0,T] \times H)$
 with $\partial_xf \in UC([0,T] \times H, D(A^*))$. Then,  $\mathbb{P}-a.s.$,
\begin{multline}
\label{eq:first-Dinkin-pre}
f(t,\X(t)) =  f(0,x)  + \int_0^t \partial_s f(r,\X(r)) \ud r  +  \int_0^t \left\langle A^* \partial_x f(r,\X(r)), \X(r) \right\rangle \ud r
 +  \int_0^t \left\langle \partial_x f(r,\X(r)), b(r) \right\rangle \ud r\\
+ \frac{1}{2}  \int_0^t \text{Tr} \left [\left ( \sigma (r) {Q}^{1/2} \right )
\left ( \sigma (r) Q^{1/2}  \right)^* \partial_{xx}^2 f(r,\X(r)) \right ] \ud r
+ \int_0^t \left\langle \partial_x f(r,\X(r)), \sigma(r) \ud \W_Q(r) \right\rangle,
\end{multline}
where the partial derivative $\partial^2_{xx} f(r, x)$ for any $r \in [0,T]$ and $x \in H$ stands in fact for its associated linear bounded operator
in the sense of \eqref{eq:expressionLB}.
\end{Corollary}

\begin{proof}

It is a consequence of Theorem \ref{th:exlmIto} taking into account Lemma \ref{lm:per-lemmaIto}:
we have  $\M(t) = x + \int_0^t \sigma (r) \ud \W_Q(r), t \in [0,T]$,
$\V(t)  = \int_0^t b(r) dr$, $\S = \Y$ with  
$\Z(r) = \X(r)$. 
According to that lemma, in   Theorem  \ref{th:exlmIto}
 we set $C = [\M,\M]^{cl}$. We also use the chain rule
for It\^o's integrals in Hilbert spaces, see the considerations
before  Proposition
\ref{PChainRule},  together with Lemma \ref{lm:lemmaDPZ-covariation}.
The fourth integral in the right-hand side of 
\eqref{eq:first-Dinkin-pre} appears from the second integral
in \eqref{eq:Ito} using Proposition \ref{PPTTT}
and  Lemma \ref{lm:lemmaDPZ-covariation}.
\end{proof}

%
%
%
%

\subsection{Extension to fractional processes}
The class of convolution processes on which the concepts of Section \ref{SecTensor} can be applied can be extended to include fractional processes. We sketch this fact in a simplified framework. We consider again $H$ a separable Hilbert space and $A$ the generator of a $C_0$-semigroup on $H$.

Let $\B\colon [0,T] \times \Omega \to H$ be a continuous progressively measurable $H$-valued process with a finite number of modes of the form
\[
\B(t) := \sum_{i=1}^N h_i \beta_i(t),
\]
where, for any $i$, $h_i$ is a fixed elements of $H$ and $\beta_i$ is a $\gamma$-H\"older continuous real-valued process, $\gamma>\frac{1}{2}$. We suppose that, for any $i= 1,.., N$, the $H$-valued function $t\mapsto e^{tA}h_i$ (defined on $[0,T]$), is $\alpha$-H\"older continuous with values in $H$. By Proposition \ref{P25}
\[
\int_0^t e^{(t-s)A} d^-\B(s) \qquad \text{and} \qquad \int_0^t e^{(t-s)A} d^+\B(s)
\]
are well-defined and equal
\[
\sum_{i=1}^N \int_0^t e^{(t-s)A} h_i d^y \beta_i(s).
\]

\begin{Example}
Suppose that $A$ is the generator of an analytic semigroup on $H$ and that $0$ is in the resolvent of $A$. It is the case for instance of the heat semigroup with Dirichlet boundary condition acting on bounded and regular domains in $\mathbb{R}^m$. If we choose $h_i\in D((-A)^{\alpha})$ for $\alpha \in (0,1]$, we have (see e.g. \cite{Pazy83}, Theorem 6.13, page 74), for any $s,t \in [0,T]$ with $s<t$, $\left | (e^{tA} - e^{sA}) h_i \right | \leq C_T (t-s)^{\alpha} | h_i |_{D((-A)^{\alpha})}$ and then $t\mapsto e^{tA}h_i$ is $\alpha$-H\"older continuous.
\end{Example}

In the setting described above consider the process 
\[
\X_1 (t) := \X(t) + \int_0^t e^{(t-s)A} d^-\B(s),
\]
where $\X$ is of the form (\ref{eq:SPDE-mild}). We show now that Corollaries \ref{cor:X-barchi-Dirichlet} and \ref{lm:Ito-pre} also hold in the present extended framework.
\begin{Proposition}
The process $\X_1$ is a $\bar \chi$-Dirichlet process.
 Moreover it
 is also a $\bar\chi$ finite quadratic variation process
 and a $\bar\nu_0\hat\otimes_\pi \mathbb{R}$-weak-Dirichlet process.
\end{Proposition}
\begin{proof}[Sketch of the proof]
Using the properties of the Young integral one can show that the process $\A_1(t) := \int_0^t e^{(t-s)A} d^-\B(s)$ is H\"older continuous with parameter $\gamma>1/2$. So $\A_1$ has  a zero scalar quadratic variation and then, by additivity and Corollary \ref{cor:X-barchi-Dirichlet}, $\X_1$ has a $\bar\chi$ quadratic variation and
$[\X_1, \X_1]_{\bar\chi} = [\X,\X]_{\bar\chi}$. A similar argument shows that $\X_1$ is a  $\bar\nu_0\hat\otimes_\pi \mathbb{R}$-weak-Dirichlet process.
\end{proof}

\begin{Proposition}
\label{pr:last}
Suppose that hypotheses of Corollary \ref{lm:Ito-pre} hold. Suppose that the trajectories of the  process $b$ belong a.s. to $L^p(0,T; H)$ for any $p<2$. Then the formula (\ref{eq:first-Dinkin-pre}) still holds with the addition of the term $\int_0^t \left\langle \partial_x f(s, \X_1(s)), d^-\B(s) \right\rangle$
\end{Proposition}
\begin{Remark}
Observe that $\int_0^t \left\langle \partial_x f(s, \X_1(s)), d^-\B(s) \right\rangle$ equals $\sum_{i=1}^N \int_0^t \partial_x f(s, \X_1(s)) h_i d^y \beta_i(s)$. Indeed $\int_0^\cdot \sigma(s) \ud \W(s)$ and $\V$ are $\alpha$-H\"older continuous for any $\alpha < 1/2$. So $\X$ is $\alpha$-H\"older continuous for any $\alpha < 1/2$ and the same holds for $\partial_x f(\cdot, \X_1(\cdot))$ since $\partial_x f$ is locally Lipschitz-continuous.
\end{Remark}
\begin{proof}[Sketch of the proof of Proposition \ref{pr:last}]
We define $\Y_1 (t) := \Y(t) + \B(t)$ where $\Y(t)$ is defined in (\ref{eq:defY}). Applying Lemma \ref{lm:conOndrejat} to $\Y^\varepsilon (t) := \Y(t) - \int_0^t \frac{\B(s) - \B(s - \varepsilon)}{\varepsilon} \ud s$ we obtain, for any $z\in D(A^*)$, $\left\langle \Y^\varepsilon(t), z \right\rangle = \int_0^t \left\langle \X^\varepsilon(s), A^*z \right\rangle \ud s$ where $\X^\varepsilon (t) := \X(t) + \int_0^t e^{(t-s)A}  \frac{\B(s) - \B(s - \varepsilon)}{\varepsilon} \ud s$. Using the fact that $\int_0^t e^{(t-s)A} d^+ \B(s) = \lim_{\varepsilon \to 0} \int_0^t e^{(t-s)A}  \frac{\B(s) - \B(s - \varepsilon)}{\varepsilon} \ud s$ the statement of Lemma \ref{lm:conOndrejat} extends to $\X_1$, $\Y_1$ replacing, respectively, $\X$ and $\Y$. Then, by Proposition \ref{pr:Y-zero-chi-quad-var}, $\Y_1$ has a zero $\bar\chi$ quadratic variation. 
The proof can be completed performing similar arguments as those in  the proof of Corollary \ref{lm:Ito-pre}.
\end{proof}

\color{black}

\begin{small}

\bibliographystyle{plain}

\bibliography{biblio}

\appendix

\section{Some functional analysis results}
\label{sec:preliminaries}

\subsection{Notation and basic facts}
\label{SBN}
Let us consider two real separable Banach spaces ${B_1}$ and ${B_2}$. 
We denote by $C(B_1; B_2)$ the set of the
continuous functions from  ${B_1}$ to ${B_2}$. 
It is a topological vector space if equipped with topology of the uniform convergence on compact sets.  If  ${B_2}=\mathbb{R}$ we often simply use the notation $C({B_1})$
 instead of $C({B_1};\mathbb{R})$. Similarly, given a real interval $I$,
 continuous ${B_2}$-valued
 functions defined on $I\times {B_1}$ while we use the lighter notation 
$C(I \times {B_1})$ when ${B_2}=\mathbb{R}$.
 $C^1(I \times {B_1})$ denotes the space of 
Fr\'echet continuous differentiable functions 
$u: I \times B_1 \rightarrow \R$.
For a function $u: I \times B_1 \rightarrow \R$, we denote by $\partial _x u(t,x) $ (respectively $\partial^2_{xx}
 u(t,x) $), if it exists, the first (respectively second) Fr\'echet derivative  with respect to  the variable $x \in {B_1}$). A function $u \in C(I \times {B_1})$  (respectively $u  \in C^1(I \times {B_1})$) will be said to belong to  $C^{0,1}(I \times {B_1})$ (respectively $C^{1,2}(I \times {B_1})$) if $\partial _x  u$  exists and it is continuous, i.e. it belongs to $C(I \times {B_1};{B_1}^*)$  (respectively  $\partial^2_{xx} u (t,x$) exists for any $(t,x) \in [0,T] \times B_1$ and it is continuous, i.e.  it belongs to $C(I \times {B_1}; {\mathcal Bi}({B_1},{B_1}))$, where  ${\mathcal Bi}({B_1},{B_2}))$ is the linear topological
space of bilinear bounded forms on $B_1 \times B_2$. 

We denote by $\mathcal{L}({B_1}; {B_2})$ the space of linear bounded maps from ${B_1}$ to ${B_2}$ and  by $\Vert \cdot \Vert_{\mathcal{L}({B_1}; {B_2})}$ the corresponding norm. We indicate by a double bar, i.e. $\Vert \cdot \Vert$, the norm of an operator. Often we will consider the case of two separable Hilbert spaces $U$ and $H$. We denote $|\cdot|$ and $\left\langle \cdot , \cdot \right\rangle$ (respectively $|\cdot|_U$ and $\left\langle \cdot , \cdot \right\rangle_U$) the norm  and the inner product on $H$ (respectively $U$).

\begin{Notation} \label{NotHilbert}
If $H$ is a Hilbert space, in order to argue more transparently, we  often distinguish between $H$ and its dual $H^*$ and with every element $h\in H$ we associate $h^*\in H^*$ through Riesz Theorem.
\end{Notation}


If $U= H$, we set $\shl(U) := \shl(U;U)$. $\shl_2(U;H)$ will be the set of {\it Hilbert-Schmidt} operators from $U$ to $H$ and $\shl_1(H)$ (respectively $\shl_1^+(H)$) will be the space of (respectively non-negative)
{\it nuclear} operators on $H$. For  details about the notions of Hilbert-Schmidt and nuclear operator, 
the reader may consult \cite{Ryan02}, Section 2.6 and \cite{PeszatZabczyk07} Appendix A.2. If $ T \in \shl_2(U;H)$ and $T^\ast: H\rightarrow U$ is the adjoint operator, then $T T^\ast \in \shl_1(H)$ and the Hilbert-Schmidt
norm of $T$ is $\Vert T \Vert^2_{\shl_2(U;H)} = \Vert T T ^\ast \Vert_{\shl_1(H)}.$
We recall that, for a generic element $T\in {\mathcal L}_1(H)$ and given a basis $\left \{ e_n \right \}$ of $H$, the sum
$
\sum_{n=1}^{\infty} \left \langle T e_n, e_n  \right \rangle
$
is absolutely convergent and independent of the chosen basis 
$\left \{ e_n \right \}$. It is called \emph{trace} of $T$
 and denoted by ${\mathrm Tr}(T)$. 
$\shl_1(H)$ is a Banach space and we denote by $\Vert \cdot \Vert_{\shl_1(H)}$
the corresponding norm.
If $T$ is non-negative then ${\mathrm Tr}(T) = \| T \|_{{\mathcal L}_1 (H)}$ and
 in general we have the inequalities
\begin{equation}
\label{eq:trace-h1}
|{\mathrm Tr}(T)| \leq \| T \|_{{\mathcal L}_1 (H)}, \quad
\sum_{n=1}^{\infty} |\left \langle T e_n, e_n  \right \rangle| \leq \| T
 \|_{{\mathcal L}_1 (H)},
\end{equation}
see Proposition C.1,
\cite{DaPratoZabczyk92}.
As a consequence,  if $T$ is a non-negative operator,  the  relation below holds:
\begin{equation} \label{EHSQ}
\Vert T \Vert_{\shl_2(U;H)}^2 = {\mathrm Tr} (T T^\ast).
\end{equation}

\subsection{Projective norms on tensor products}
\label{SRN}

Consider two real separable Banach spaces $B_1$ and $B_2$. Denote, for $i=1,2$, with 
$|\cdot|_{B_i}$ the norm on $B_i$.
$B_1\otimes B_2$ stands for the \emph{algebraic tensor product} i.e.
 the set of the elements of the form $\sum_{i=1}^n x_i\otimes y_i$ where
 $x_i$ and  $y_i$ are respectively elements of $B_1$ and $B_2$. 
On $B_1\otimes B_2$ we identify all the expressions we need in order to ensure 
that the product $\otimes\colon B_1\times B_2 \to B_1\otimes B_2$ is bilinear.

On $B_1\otimes B_2$ we introduce the \emph{projective norm $\pi$} defined, for all $u\in B_1\otimes B_2$, as
\[
\pi(u) := \inf \left \{ \sum_{i=1}^n  |x_i|_{B_1} |y_i|_{B_2} \; : \; 
u = \sum_{i=1}^{n}  x_i \otimes y_i  \right \}.
\]
The \emph{projective tensor product of $B_1$ and $B_2$}, $B_1\hat\otimes_\pi B_2$, is the Banach space obtained as completion of $B_1\otimes B_2$ for the norm $\pi$, see \cite{Ryan02} Section 2.1, or \cite{DiGirolamiRusso09} for further details.  
%
%
For any $x\in B_1$ and $y\in B_2$ one has $\pi(x\otimes y) = |x|_{B_1} |y|_{B_2}$, see \cite{Ryan02} Chapter 6.1 for details.

\begin{Notation} \label{Not1}
When $B_1=B_2$ and $x\in B_1$ we denote  by  $x^{\otimes 2}$ the element $x\otimes x \in B_1 \otimes B_1$. 
\end{Notation}
The first natural example of element of the dual space $(B_1 \otimes B_2)^*$
can be constructed with a couple $a^*$ and $b^*$ respectively elements
of $B_1^*$ and $B_2^*$.

\begin{Lemma}
\label{lm:che-era-remark}
Let $B_1$ and $B_2$ be two real separable Banach spaces.
 We denote  $B := B_1\hat\otimes_\pi B_2$.
Choose $a^* \in B_1^*$ and  $b^* \in B_2^*$. 
One can associate with $a^*\otimes b^*$ the element $\ell$
of $B^*$ , denoted by $i(a^*\otimes b^*)$, 
acting as follows on a generic element $u = \sum_{i=1}^{n}  x_i \otimes y_i \in B_1\otimes B_2$:
\begin{equation}
\label{eq:defiastarbstar}
\white{\rangle}_{B^*}\left\langle i(a^*\otimes b^*), u \right \rangle_{B} = \sum_i^n \white{\rangle}_{B_1^*} \left\langle a^*, x_i \right\rangle_{B_1} \!\!
\white{\rangle}_{B^*_2} \left\langle b^*, y_i \right\rangle_{B_2}.
\end{equation}
Then $i(a^*\otimes b^*)$ extends by continuity to the whole
 $B$ and  
\begin{equation} \label{TensEq}
 \vert i(a^*\otimes b^*) \vert_{B^*}  =  |a^*|_{B_1^*}|b^*|_{B_2^*}.
\end{equation}



\end{Lemma}
\begin{proof}
We first prove the $\leq$  inequality in \eqref{TensEq}. 
Observe first that the functional $i(a^* \otimes b^*)$ in (\ref{eq:defiastarbstar}) is well-defined thanks to the universal property of the tensor product, in particular if $\sum_{i=1}^{\tilde n}  \tilde x_i \otimes \tilde y_i$ is another representation of $u$ we have $\sum_i^{\tilde n}\!\! \white{\rangle}_{B_1^*}\left\langle a^*, \tilde x_i \right\rangle_{B_1} \!\! 
\white{\rangle}_{B^*_2} \left\langle b^*, \tilde y_i \right\rangle_{B_2} = \sum_i^n\!\! 
\white{\rangle}_{B_1^*} \left\langle a^*, x_i \right\rangle_{B_1}\!\! \white{\rangle}_{B^*_2} 
\left\langle b^*, y_i \right\rangle_{B_2}$. Since, for any 
possible representation, we have
\[
\left | \!\! \white{\rangle}_{B^*} \left\langle i(a^*\otimes b^*), u \right \rangle_{B} \right | \leq \sum_i^{\tilde n}
|a^*|_{B_1^*} |\tilde x_i|_{B_1} |b^*|_{B_2^*} |\tilde y_i|_{B_2},
\]
the $\leq$  inequality in \eqref{TensEq} follows by taking the infimum on the possible representations.

Concerning the converse inequality, we have
$|a^*|_{B_1^*} = \sup_{|\phi|_{B_1}=1} \white{\rangle}_{B_1^*\!}\left\langle a^*, \phi \right\rangle_{B_1}$ and similarly for $b^*$. So, chosen $\delta>0$, there exist $\phi_1\in B_1$ and $\phi_2\in B_2$ with $|\phi_1|_{B_1}=|\phi_2|_{B_2}=1$ and
\[
|a^*|_{B_1^*} \leq \delta + \white{\rangle}_{B_1^*\!}\left\langle a^*, \phi_1 \right\rangle_{B_1}, \qquad
|b^*|_{B_2^*} \leq \delta + \white{\rangle}_{B_2^*\!}\left\langle b^*, \phi_2 \right\rangle_{B_2}.
\]
We set $u:= \phi_1\otimes \phi_2$. We obtain
\begin{multline}
|i(a^*\otimes b^*)|_{B^*} \geq \frac{\white{\rangle}_{B^*\!}\left\langle i(a^*\otimes b^*), u \right\rangle_{B}}{|u|_{B}}=
\frac{\white{\rangle}_{B^*\!}\left\langle i(a^*\otimes b^*), u \right\rangle_{B}}{|\phi_1|_{B_1}|\phi_2|_{B_2}}\\
 = 
\white{\rangle}_{B_1^*\!}\left\langle a^*, \phi_1 \right\rangle_{B_1}  \white{\rangle}_{B_2^*\!}\left\langle 
 b^*, \phi_2 \right\rangle_{B_2}
\geq (|a^*|_{B_1^*} - \delta) (|b^*|_{B_2^*} - \delta).
\end{multline}
Since $\delta>0$ is arbitrarily small we finally obtain $|i(a^*\otimes b^*)|_{B^*} \geq |a^*|_{B_1^*} |b^*|_{B_2^*}$. This gives the second inequality and concludes the proof.
\end{proof}

The dual of the projective tensor product  $B_1\hat\otimes_\pi B_2$,
denoted by $(B_1\hat\otimes_\pi B_2)^*$, can be identified isomorphically
with the linear space of bounded bilinear forms on $B_1\times B_2$
 denoted  by   ${\mathcal Bi}(B_1,B_2)$. If $\ell \in (B_1\hat\otimes_\pi B_2)^*$
and $\psi_\ell $ is the associated form in ${\mathcal Bi}(B_1,B_2)$,
we have
\begin{equation} \label{E1} 
\vert \ell \vert_{(B_1\hat\otimes_\pi B_2)^*} = \sup_{\vert x \vert_{B_1} \le 1, 
\vert y \vert_{B_2} \le 1} \vert \psi_\ell(x,y) \vert. 
\end{equation}
See for this  \cite{Ryan02} Theorem 2.9  Section 2.2, page 22 and also
 the discussion after the proof of the theorem, page 23. 

\begin{Remark} \label{RIdentification}
We illustrate the mentioned identification in a particular case, see \cite{Ryan02} Chapter 4 for more.
Let $a^* \in B_1, b^* \in B_2$  as in Lemma \ref{lm:che-era-remark}
If $\ell = i(a^* \otimes b^*)$ then 
$\psi_\ell\left (x, y \right ):=  \left\langle a^*, x 
\right\rangle \left\langle b^*, y \right\rangle.$ More generally, 
by bilinearity, let 
$\ell=\sum_{i=1}^n a_i^*\otimes b_i^*$ for some $a_i^*\in B_1^*$ and $b_i^* \in B_2^*$. 
$i(\ell)$ can be identified with the element $\psi_\ell$ of
 ${\mathcal Bi}(B_1,B_2)$ acting as 
\begin{equation}\label{E10}
\psi_\ell\left (x, y \right ):= \sum_{i=1}^n \left\langle a_i^*, x 
\right\rangle \left\langle b_i^*, y \right\rangle.
\end{equation}
 \end{Remark}


We go on with consideration related to projective norms on
tensor products of Hilbert spaces.
Every  element $u \in H\hat\otimes_\pi H$  is isometrically associated 
with an element $T_u$ in the space of nuclear operators  
${\mathcal L}_1 (H,H)$, defined, for $u$ of the form
  $\sum_{i=1}^{\infty} a_n\otimes b_n$, as follows:
 \[
 T_u(x) := \sum_{i=1}^{\infty} \left\langle x, a_n \right\rangle b_n,
 \]
see for instance \cite{Ryan02} Corollary 4.8 Section 4.1 page 76.

$T_u$ is self-adjoint if and only if there exists a sequence of real numbers $(\lambda_n)$ and an orthonormal basis $(h_n)$ of $H$ such that
\begin{equation}
\label{eq:expressionu}
u= \sum_{n=1}^{+\infty} \lambda_n h_n \otimes h_n.
\end{equation}
The ``only if'' part is obvious. For the ``if'' part observe that, since $T_u$ is nuclear, it is compact (Proposition A.6 of \cite{PeszatZabczyk07}); if is also self-adjoint there exists (thanks to the spectral theorem) a sequence of real numbers $(\lambda_n)$ and an orthonormal basis $(h_n)$ of $H$ such that $T_u$ can be written as
\begin{equation}
\label{eq:expressionTu}
 T_u(x)= \sum_{n=1}^{+\infty} \lambda_n \left \langle h_n, x \right\rangle h_n, \qquad \text{for all $x\in H$};
\end{equation}
in particular $T_u(h_n) = \lambda_n h_n$ for each $n$ and $u$ can be written as in (\ref{eq:expressionu}).

To each element $\ell $ of $(H\hat\otimes_\pi H)^*$ we associate 
 a linear continuous operator 
$L_\ell: H \rightarrow H$ 
(see \cite{Ryan02} page 24, the discussion before Proposition 2.11 
Section 2.2) such that
\begin{equation}
\label{eq:expressionLB}
\left\langle L_\ell (x), y \right\rangle = \ell(x \otimes y) \;\;
(= \psi_\ell(x,y)) \qquad \text{for all $x, y \in H$}.
\end{equation}

\begin{Proposition} \label{PPTTT}
Let  $u \in  H\hat\otimes_\pi H$  and $\ell \in (H\hat\otimes_\pi H)^*$ 
with associated maps $T_u \in  \mathcal{L}_1(H), L_\ell \in   \mathcal{L}(H)$.
If $T_u$ is self-adjoint
\[
{}_{(H\hat\otimes_\pi H)^*} 
\langle \ell,u \rangle_{H\hat\otimes_\pi H} = {\rm Tr} \left (L_\ell T_u \right ).
\]
\end{Proposition}
\begin{proof}
The claim follows from what we have recalled above. Indeed, using (\ref{eq:expressionu}) and (\ref{eq:expressionLB}) we have
\[
{}_{(H\hat\otimes_\pi H)^*} 
\langle \ell,u \rangle_{H\hat\otimes_\pi H} = \ell(u) = 
\ell\left ( \sum_{i=1}^{+\infty} \lambda_n h_n \otimes h_n \right ) = 
\sum_{n=1}^{+\infty} \left \langle L_\ell (\lambda_n h_n), h_n \right\rangle
\]
and the last expression is exactly ${\rm Tr} \left (L_\ell T_u \right )$ when
 we compute it using the basis $h_n$.
\end{proof}

\begin{Lemma}
\label{lm:era52}
Let us consider a Banach space $\nu_1$ (respectively $\nu_2$) continuously embedded in $B_1^*$ (respectively $B_2^*$).  Then $\bar\chi := \nu_1 \hat\otimes_\pi \nu_2$ can be continuously embedded in
$(B_1\hat\otimes_\pi B_2)^*$.  In particular there exists a constant $C>0$
 such that,
for all $u\in \bar\chi$,
\begin{equation}
\label{eq:condizione-chi}
|\ell|_{(B_1\hat\otimes_\pi B_2)^*} \leq C |\ell|_{\bar\chi},
\end{equation}
after having identified an element of $\bar\chi$ with an element
of $(B_1\hat\otimes_\pi B_2)^*$, see Lemma \ref{lm:che-era-remark}.
In other words $\bar\chi$ is a Chi-subspace of $(B_1\hat\otimes_\pi B_2)^*$. In particular $B_1^* \hat\otimes_\pi B_2^*$  is a  Chi-subspace of    $(B_1\hat\otimes_\pi B_2)^*$.
\end{Lemma}
\begin{proof}


To simplify the notations assume the norm of the injections $\nu_1 
\hookrightarrow B_1^*$
 and $\nu_2 \hookrightarrow B_2^*$ to be less or equal than $1$.
We recall that  $(B_1\hat\otimes_\pi B_2)^*$ is isometrically identified 
with the Banach space ${\mathcal Bi}(B_1,B_2)$.

Consider first an element $\ell \in \bar\chi$ of the form 
 $\ell=\sum_{i=1}^n a_i^*\otimes b_i^*$ for some $a_i^*\in \nu_1$ and
 $b_i^* \in \nu_2$. By \eqref{E10} $\ell$ can be identified with the element 
$\psi_\ell$ of ${\mathcal Bi}(B_1,B_2)$ acting as
 $\psi_\ell\left ( x, y \right ):= \sum_{i=1}^n \left\langle a_i^*, x
 \right\rangle \left\langle b_i^*, \ell \right\rangle$. 
We can choose $a_i^*\in \nu_1$ and $b_i^*\in \nu_2$ such that $u=\sum_{i=1}^n a_i^*\otimes b_i^*$ and
\[
|\ell|_{\bar\chi} = \inf \left \{ \sum_{i=1}^n |x_i|_{\nu_1} |y_i|_{\nu_2} \; : \;
 \ell  = \sum_{i=1}^{n} x_i \otimes y_i, \qquad x_i\in \nu_1,\; y_i \in \nu_2  \right \}
 > -\epsilon + \sum_{i=1}^n |a_i^*|_{\nu_1} |b_i^*|_{\nu_2}.
\]
Using such an expression for $\ell$ we have 
\[
\| \psi_\ell \|_{{\mathcal Bi}(B_1,B_2)} = \sup_{|\phi|_{B_1}, |\psi|_{B_2} \leq 1} \left | \sum_{i=1}^n 
\white{\rangle}_{{B_1^*}}\left\langle a_i^*, \phi \right\rangle_{{B_1}} \white{\rangle}_{{B_2^*}}\left\langle b_i^*, \psi
 \right\rangle\white{\rangle}_{{B_2}} \right | 
\leq \sum_{i=1}^n |a_i^*|_{B_1^*} |b_i^*|_{B_2^*}
 \leq \sum_{i=1}^n |a_i^*|_{\nu_1} |b_i^*|_{\nu_2} \leq \epsilon + |\ell|_{\bar\chi}.
\]
Since $\epsilon$ is arbitrary, we conclude
 that $\| \ell \|_{{\mathcal Bi}(B_1,B_2)} \leq  |\ell|_{\bar\chi}$.

Since this proves that the  mapping  that associates to 
 $\ell \in \nu_1 \hat\otimes_\pi \nu_2$ its corresponding element in
 ${\mathcal Bi}(B_1,B_2)$, has norm $1$ on the dense subset 
$\nu_1 \hat\otimes_\pi \nu_2$, then the claim is proved.
\end{proof}

\begin{Lemma}
 \label{lemmaTensor} 
Let $H_1, H_2$ be two separable Hilbert spaces. Then
$H_1^*\hat\otimes_\pi H_2^*$ is sequentially dense in $(H_1\hat\otimes_\pi H_2)^*$
 in the weak-* topology.
\end{Lemma}
\begin{proof} 

Let $(e_i)$ and $(f_i)$ be respectively two orthonormal bases of $H_1$ 
and $H_2$.
  We denote  by $\mathcal{D}$ the linear span of finite linear combinations
 of $e_i \otimes f_i$.
 Let $\ell \in (H_1\hat\otimes_\pi H_2)^*$, which is a linear continuous functional on $H_1\hat\otimes_\pi H_2$.
 Taking into account the identification \eqref{E10}
 of $(H_1\hat\otimes_\pi H_2)^*$ with ${\mathcal Bi}(H_1, H_2)$,
 for each
 $n\in \mathbb{N}$, we define the continuous bilinear form
\begin{equation} \label{E20}
\psi^n(x,y) := \sum_{i=1}^{n} \left\langle x, e_i \right\rangle_{H_1}
  \left\langle y, f_i \right\rangle_{H_2} \psi_\ell(e_i, f_i),  \ x \in H_1,
y \in H_2.
\end{equation}
By \eqref{E10} it defines an element $\ell_n$ of $(H_1\hat\otimes_\pi H_2)^*$
 such that 
$\psi_{\ell_n} = \psi_n$ which coincides (via the application $i$) 
with $\sum_{i=1}^n(e_i^* \otimes f_i^*)$
 of $H_1^*\hat\otimes_\pi H_2^* \subset  (H_1\hat\otimes_\pi H_2)^*$. It remains 
to show that 
\[
\tensor[_{(H_1\hat\otimes_\pi H_2)^*}]{\big\langle \ell_n, l \big\rangle}
{_{H_1\hat\otimes_\pi H_2}} \xrightarrow{n\to\infty} \tensor[_{(H_1\hat\otimes_\pi H_2)^*}]{\big\langle \ell, l \big\rangle}{_{H_1\hat\otimes_\pi H_2}}, 
\quad \text{for all $l\in H_1\hat\otimes_\pi H_2$}.
\]
We show now the following.
\begin{itemize}
 \item[(i)] $\psi_n(x,y) \xrightarrow{n\to\infty} \psi_\ell(x,y)$
 for all $x\in H_1$, $y\in H_2$.
 \item[(ii)] For a fixed $l\in H_1\hat\otimes_\pi H_2$, the sequence $\ell_n(l)$ is bounded.
\end{itemize}
Let us prove first $(i)$. Let $x\in H_1$ and $y\in H_2$. Using \eqref{E20}
\begin{equation}
\label{eq:TE1}
\psi^n(x,y) = \psi_\ell \left ( \sum_{i=1}^n \left\langle x, e_i\right\rangle_{H_1} e_i,  \quad \sum_{i=1}^n \left\langle y, f_i\right\rangle_{H_2} f_i \right ).
\end{equation}
Since $\sum_{i=1}^n \left\langle x, e_i\right\rangle e_i \xrightarrow{n\to +\infty} x$ in $H_1$,
$\sum_{i=1}^n \left\langle y, f_i\right\rangle f_i \xrightarrow{n\to +\infty} y$ in $H_2$, and $\psi_\ell$ is a bounded bilinear form, the point (i) follows.

Let us prove now $(ii)$. Let $\epsilon>0$ fixed and
 $l_0 \in 
\mathcal{D}$ 
such that $|l-l_0|_{H_1\hat\otimes_\pi H_2} \leq \epsilon$. Then 
\begin{equation}
\label{eq:TE2}
|\ell_n(l)| \leq |\ell_n(l-l_0)| + |\ell_n(l_0)| \leq 
|\ell_n|_{(H_1\hat\otimes_\pi H_2)^*} |l-l_0|_{H_1\hat\otimes_\pi H_2} + |\ell_n(l_0)|.
\end{equation}
So (\ref{eq:TE2})  is bounded by 
\begin{eqnarray*}
\sup_{\vert x \vert_{H_1}, \vert y\vert_{H_2} \le 1} \sum_{i=1}^n 
\vert \left\langle x, e_i\right\rangle e_i \vert_{H_1}
\vert \left\langle y, f_i\right\rangle f_i \vert_{H_2}
|\ell|_{(H_1\hat\otimes_\pi H_2)^*}  \epsilon + \sup_{n} |\ell_n(l_0)|
\le |\ell|_{(H_1\hat\otimes_\pi H_2)^*}  \epsilon + \sup_{n} |\ell_n(l_0)|,
\end{eqnarray*}
recalling that the sequence $(\ell_n(l_0))$ is bounded, since it is convergent.
Finally (ii) is also proved. \\
At this point (i) implies that
\[
\tensor[_{(H_1\hat\otimes_\pi H_2)^*}]{\big\langle \ell_n, 
l \big\rangle}{_{H_1\hat\otimes_\pi H_2}} \xrightarrow{n\to\infty} \tensor[_{(H_1\hat\otimes_\pi H_2)^*}]{\big\langle \ell, l \big\rangle}{_{H_1\hat\otimes_\pi H_2}}, \quad \text{for all $l\in \mathcal{D}$}.
\] 
Since $\shd$ is dense in $H_1\hat\otimes_\pi H_2$ , the conclusion follows by 
 Banach-Steinhaus theorem, see Theorem 18, Chapter II in
 \cite{DunfordSchwartz58}.
\end{proof}

\bigskip
{\bf ACKNOWLEDGEMENTS.}
The authors are grateful to the Editor, Associated Editor and to both Referees
for their valuable comments which helped them to conisderably improve the
first version of the paper.
 It was partially written during the stay of the
second named author in the Bernoulli Center (EPFL Lausanne)
and at Bielefeld University, SFB 701 (Mathematik).\\
    The research was partially 
 supported by the 
ANR Project MASTERIE 2010 BLAN-0121-01.
The second named author also benefited partially from the support of the
 ``FMJH Program Gaspard Monge in optimization and operation research'' 
(Project 2014-1607H). 
The work of the first named author has been developed in the framework of
 the center
 of excellence LABEX MME-DII (ANR-11-LABX-0023-01).

\end{small}

\end{document}